\def\siam{0}

\if\siam1
\documentclass{siamltex1213}
\else
\documentclass{article}
\fi

\if\siam1

\newtheorem{theoreml}{Theorem}

\newtheorem{example}{Example}
\newtheorem{remark}{Remark}

\usepackage{changebar}

\else

\usepackage[left=1in,right=1in,top=1in,bottom=1in]{geometry}
\usepackage[width=12cm,format=plain,labelfont=sf,bf,small,textfont=sf,up,small,justification=justified,labelsep=period]{caption}

\usepackage{amsthm}
\usepackage{graphicx}
\usepackage[backref=page]{hyperref}

\theoremstyle{definition}
\newtheorem{definition}{Definition}
\newtheorem{example}{Example}
\newtheorem{remark}{Remark}

\theoremstyle{plain}
\newtheorem{proposition}{Proposition}

\newtheorem{lemma}{Lemma}
\newtheorem{theorem}{Theorem}
\newtheorem*{theorem*}{Theorem}
\newtheorem{theoreml}{Theorem}

\hypersetup{
    colorlinks=true,       
    linkcolor=blue,          
    citecolor=blue,        
    filecolor=blue,      
    urlcolor=blue           
}

\usepackage[outerbars,color]{changebar}

\fi

\newenvironment{reptheorem}[1]{\bgroup\renewcommand{\thetheorem}{\ref*{#1}}\begin{theorem}}{\end{theorem}\addtocounter{theorem}{-1}\egroup}

\usepackage{amsmath}
\usepackage{amssymb}
\usepackage{amsfonts}

\usepackage{xcolor}
\usepackage{tabularx}

\newcommand{\cF}{\mathcal F}

\newcommand{\cM}{\mathcal M}

\newcommand{\RR}{\mathbb{R}}

\newcommand{\ZZ}{\mathbb{Z}}

\definecolor{darkgreen}{rgb}{0,0.8,0}

\DeclareMathOperator{\trace}{{\bf trace}}
\DeclareMathOperator{\conv}{conv}

\DeclareMathOperator{\linspan}{span}

\DeclareMathOperator{\Aut}{Aut} 
\DeclareMathOperator{\id}{id} 
\DeclareMathOperator{\1}{{\mathbf 1}}

\DeclareMathOperator{\Pol}{Pol}

\let\TH\relax
\DeclareMathOperator{\TH}{TH}

\if\siam0
\DeclareMathOperator{\diag}{diag}

\fi

\renewcommand{\imath}{\boldsymbol{i}}

\newcommand{\CUT}{\text{CUT}}
\newcommand{\PAR}{\text{PAR}} 
\newcommand{\EVEN}{\text{EVEN}} 

\newcommand{\parity}{\text{parity}}
\newcommand{\cube}{\text{cube}}

\newcommand{\fS}{\mathfrak{S}} 

\renewcommand{\S}{\mathbf{S}} 

\newcommand{\isom}{\cong} 

\renewcommand{\tilde}[1]{\widetilde{#1}}

\title{Equivariant semidefinite lifts
\if\siam1
\\
\fi
 and sum-of-squares hierarchies}

\author{Hamza Fawzi \and James Saunderson \and Pablo A. Parrilo\thanks{The authors are with
    the Laboratory for Information and Decision Systems, Department of
    Electrical Engineering and Computer Science, Massachusetts
    Institute of Technology, Cambridge, MA 02139. Email:
    \texttt{\{hfawzi,jamess,parrilo\}@mit.edu}. This research was funded in part by AFOSR FA9550-11-1-0305 and AFOSR FA9550-12-1-0287.}}
\renewcommand\footnotemark{}
\if\siam0
\date{April 15, 2015}
\fi

\begin{document}

\maketitle

\if\siam1
\slugger{siopt}{xxxx}{xx}{x}{x--x}
\pagestyle{myheadings}
\thispagestyle{plain}
\markboth{}{}
\fi

\begin{abstract}
A central question in optimization is to maximize (or minimize) a linear function over a given polytope $P$. To solve such a problem in practice one needs a concise description of the polytope $P$. In this paper we are interested in representations of $P$ using the positive semidefinite cone: a \emph{positive semidefinite lift} (psd lift) of a polytope $P$ is a representation of $P$ as the projection of an affine slice of the positive semidefinite cone $\S^d_+$. Such a representation allows linear optimization problems over $P$ to be written as semidefinite programs of size $d$. Such representations can be beneficial in practice when $d$ is much smaller than the number of facets of the polytope $P$. In this paper we are concerned with so-called \emph{equivariant} psd lifts (also known as \emph{symmetric} psd lifts) which respect the symmetries of the polytope $P$.

We present a representation-theoretic framework to study equivariant psd lifts of a certain class of symmetric polytopes known as \emph{orbitopes}. Our main result is a \emph{structure theorem} where we show that any equivariant psd lift of size $d$ of an orbitope is of sum-of-squares type where the functions in the sum-of-squares decomposition come from an invariant subspace of dimension smaller than $d^3$. We use this framework to study two well-known families of polytopes, namely the parity polytope and the cut polytope, and we prove exponential lower bounds for equivariant psd lifts of these polytopes.
\end{abstract}



\section{Introduction}

\subsection{Preliminaries and definitions}

Linear programming is the problem of computing the maximum (or minimum) of a linear function over a polytope $P$.
In order to solve such a problem in practice one needs to have an efficient description of the polytope $P$. An important technique to obtain such efficient formulations which has received renewed attention recently is that of \emph{extended formulations} or \emph{lifts}: the idea is to represent $P$ as the projection of a higher-dimensional convex set $Q$ which has a simpler description than $P$. For the purpose of optimizing a linear function, one can work over the higher-dimensional convex set $Q$ instead of $P$: indeed if $P = \pi(Q)$ where $\pi$ is a linear projection map and $\ell$ is the linear objective function then we have: 
\begin{equation}
 \label{eq:lifted-opt}
 \max_{x \in P} \ell(x) = \max_{y \in Q} (\ell \circ \pi)(y).
\end{equation}
There are many examples where extended formulations allow us to reduce the size of optimization problems. Consider for example the $\ell_1$ ball in $\RR^n$ which requires $2^n$ linear inequalities for its description: using a simple construction one can show that this polytope is the projection of a polytope in $\RR^{2n}$ which can be described using $2n$ linear inequalities only.

When lifting the polytope $P$ to a higher-dimensional convex set $Q$, it is natural to consider convex sets $Q$ of the form $Q = K\cap L$ where $K$ is a proper cone (i.e., a convex, closed, full-dimensional and pointed cone) and $L$ is an affine subspace. Indeed in this case the optimization problem \eqref{eq:lifted-opt} over $Q$ is a \emph{conic program} \cite{ben2001lectures}. When the cone $K$ is the nonnegative orthant $K = \RR^m_+$ a $K$-lift is usually referred to as an \emph{LP lift} since the resulting optimization problem is a linear program. Another important special case, which is the main focus of this paper, is when $K$ is the cone of positive semidefinite matrices $K = \S^d_+$. In this case the lift is called a \emph{psd lift} of size $d$ and the resulting optimization problem \eqref{eq:lifted-opt} over $Q$ is a semidefinite program.

\begin{definition}
Let $P \subset \RR^n$ be a polytope. We say that $P$ has a \emph{psd lift} of size $d$ if there exist an affine subspace $L \subset \S^d$ and a linear map $\pi:\S^d\rightarrow \RR^n$ such that $P = \pi(\S^d_+\cap L)$.
\end{definition}

The polytope $P \subset \RR^n$ of interest often has a lot of symmetries, i.e., it is invariant under a certain group of transformations of $\RR^n$.  For example the square $[-1,1]^2$ in the plane is invariant under permutation of coordinates; and the regular $n$-gon is invariant under the action of the dihedral group $D_{2n}$ consisting of $n$ rotations and $n$ reflections.  For such symmetric polytopes one may be interested in lifts that ``respect'' this symmetry. In the context of linear programming, such lifts were  first studied by Yannakakis \cite{yannakakis1991expressing} where he showed that any symmetric LP lift of the matching polytope must have exponential size. In the more recent works \cite{kaibel2010symmetry,goemans2009smallest,pashkovich2009tight,gouveia2011lifts}, it was shown that the symmetry requirement can have a significant impact on the size of the smallest lifts, i.e., there are polytopes (like the permutahedron for example) where there is a large gap between the smallest LP lift and the smallest \emph{symmetric} LP lift. The recent work of Chan el al. \cite{chan2013approximate} establishes, among others, a strong connection between symmetric LP lifts and the Sherali-Adams hierarchy: it is shown that the approximation quality of any polynomial-size symmetric LP for the maximum cut problem can be achieved by a constant number of rounds of the Sherali-Adams hierarchy (in fact the paper \cite{chan2013approximate} establishes a connection between general LP formulations, possibly nonsymmetric, and the Sherali-Adams hierarchy however the results are stronger in the case of symmetric LPs).

The works cited above studied the symmetry requirement in the context of LP lifts. In this paper we are interested in \emph{psd lifts} that respect the symmetries of the polytope $P$.
 Intuitively a psd lift $P = \pi(Q)$ where $Q=\S^d_+\cap L$ respects the symmetry of $P$ if any transformation $g \in GL(\RR^n)$ which leaves $P$ invariant can be lifted to a transformation $\Phi(g) \in GL(\S^d)$ that preserves both $\S^d_+$ and $L$, and so that the following equivariance relation holds: for any $y \in Q$, $\pi(\Phi(g)y) = g\pi(y)$. Such lifts are commonly called \emph{symmetric} in the literature \cite{kaibel2010symmetry,gouveia2011lifts}. In this paper however we prefer to use instead the term \emph{equivariant} which is more descriptive and mathematically standard.
It is known that the transformations of $\S^d$ which leave the psd cone $\S^d_+$ invariant are precisely congruence transformations, see e.g. \cite[Theorem 9.6.1]{tuncel2000potential}. This motivates the following definition of \emph{equivariant psd lift} which we adopt in this paper:
\begin{definition}
\label{def:equivariantpsdlift}
Let $P \subset \RR^n$ be a polytope invariant under the action of a group $G \subset GL(\RR^n)$. Assume $P = \pi(\S^d_+\cap L)$ is a psd lift of $P$ of size $d$. The lift is called \emph{$G$-equivariant} if there is a group homomorphism $\rho : G \rightarrow GL(\RR^d)$ such that the following two conditions hold:\\
(i) The subspace $L$ is invariant under congruence by $\rho(g)$, for all $g \in G$:
\begin{equation}
  \label{eq:Linvariance}
 \rho(g) Y \rho(g)^T \in L \quad \forall g \in G, \; \forall Y \in L.
\end{equation}
(ii) The following equivariance relation holds:
\begin{equation}
 \label{eq:linequivariance}
 \pi\left(\rho(g) Y \rho(g)^T\right) = g\pi(Y) \quad \forall g \in G, \; \forall Y \in \S^d_+ \cap L.
\end{equation}
\end{definition}

Observe that the notion of equivariant lift is defined with respect to a group $G$ which leaves $P$ invariant. The group $G$ does not have to be the full automorphism group $\Aut(P)$ in general. Indeed one may be interested in lifts that preserve only a certain subset of the symmetries of $P$, but not all of them. One example we discuss in detail later is the \emph{parity polytope} which is invariant under permutation of coordinates as well as under certain sign switches. In Section \ref{sec:parity} we mention two examples of well-known lifts of the parity polytope which are equivariant with respect to one set of transformations but not the other.

\paragraph{Examples} To illustrate the definition of equivariant psd lift, we now give a simple example of an equivariant psd lift and another example of a psd lift that does not satisfy the definition of equivariance.

\begin{example}
\label{ex:square}
{\bf  An equivariant psd lift of the square $[-1,1]^2$}\\
Consider the square $[-1,1]^2$ in the plane. The following is a well-known positive semidefinite lift of $[-1,1]^2$ of size 3:
\begin{equation}
\label{eq:squarelift}
 [-1,1]^2 = \left\{(x_1,x_2) \in \RR^2 \; : \; \exists u \in \RR \; \begin{bmatrix} 1 & x_1 & x_2\\ x_1 & 1 & u\\ x_2 & u & 1 \end{bmatrix} \succeq 0 \right\}. \end{equation}
We can write this lift in standard form as $[-1,1]^2 = \pi(\S^3_+ \cap L)$ where $L$ and $\pi$ are given by:
\[ L = \{ X \in \S^3 \; : \; X_{11} = X_{22} = X_{33} = 1 \} \quad \text{ and } \quad \pi(X) = (X_{12},X_{13}) \in \RR^2. \]
Note that the intersection $\S^3_+ \cap L$ is known as the \emph{elliptope} in $\S^3$ and is illustrated in Figure \ref{fig:square_lift_elliptope}.

\begin{figure}[ht]
  \centering
  \if\siam1
    \includegraphics[width=7cm]{square_lift_elliptope.eps}
  \else
    \includegraphics[width=7cm]{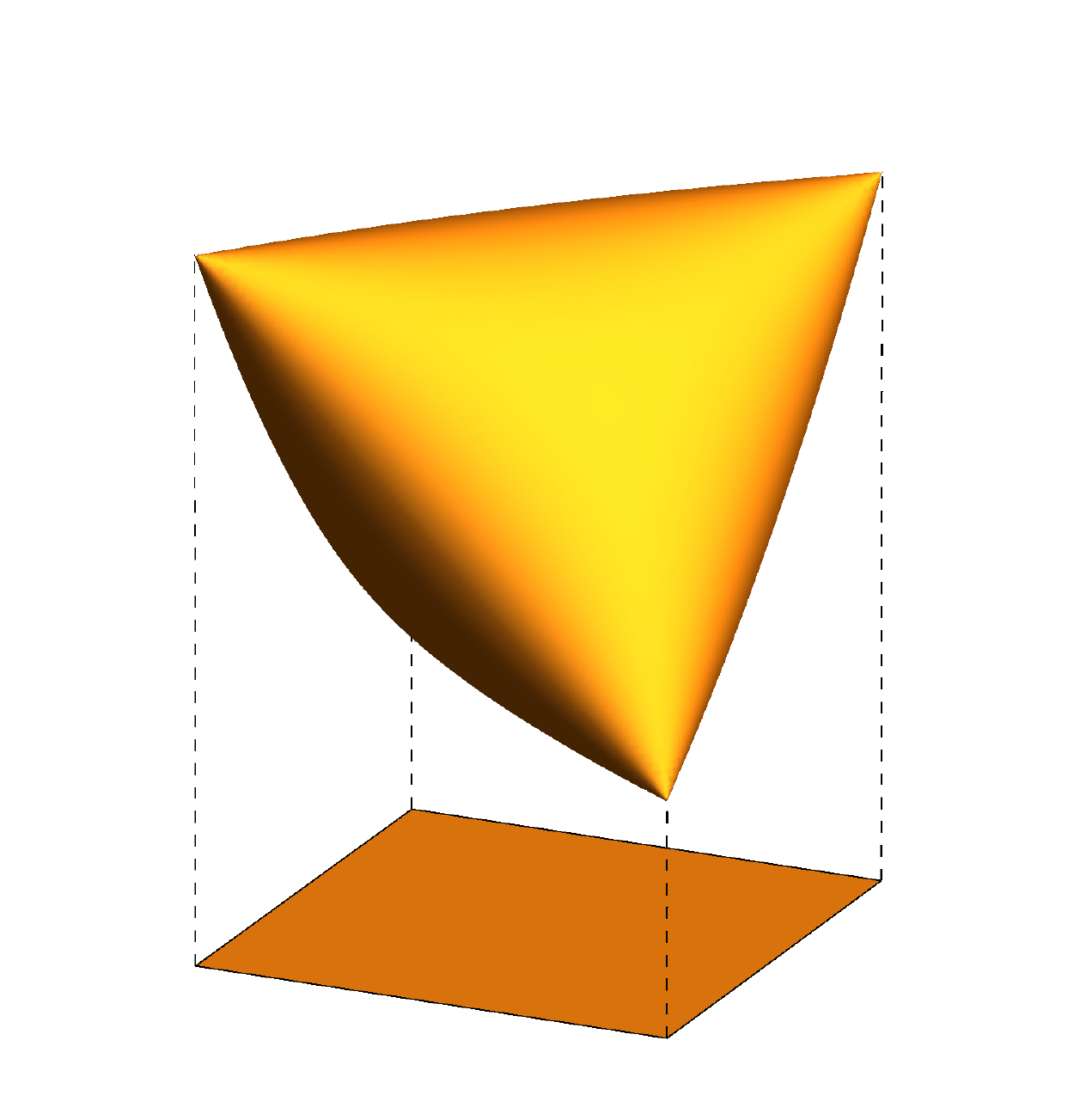}
  \fi
  \caption{A positive semidefinite lift of the square $[-1,1]^2$: the elliptope $\{ X \in \S^3_+ \; : \; \diag(X) = \1 \}$ linearly projects onto the square $[-1,1]^2$.}
  \label{fig:square_lift_elliptope}
\end{figure}

 The symmetry group of the square $[-1,1]^2$ is the dihedral group of order 8, denoted $D_8$. To show that the lift \eqref{eq:squarelift} is $D_8$-equivariant, consider the group homomorphism $\rho: D_8 \rightarrow GL(\RR^3)$ defined by:
\[ \rho(g) = \begin{bmatrix} 1 & 0\\ 0 & g \end{bmatrix} \quad \forall g \in D_8. \]
It is easy to see that the congruence operation by $\rho(g)$ stabilizes the subspace $L$, and that the following equivariance relation holds:
\[ \pi(\rho(g)X\rho(g)^T) = g \pi(X) \quad \forall g \in D_8, \forall X \in \S^3_+ \cap L. \]
Thus this shows that the psd lift \eqref{eq:squarelift} is $D_8$-equivariant. 
\if\siam1 \qquad \fi
$\lozenge$
\end{example}

We show later in the paper that any psd lift obtained from the Lasserre/theta-body hierarchy \cite{lasserre2009convex,gouveia2010theta} is in fact equivariant. For example, one concrete such psd lift is that of the stable set polytope for perfect graphs, which was originally introduced by Lov{\'a}sz \cite{lovasz1979shannon} and which can be understood as an instance of the Lasserre/theta-body lifts, see \cite[Section 3.1]{gouveia2010theta}.

\bigskip
We now show an example of a non-equivariant psd lift.
\bigskip

\begin{example}{\bf A nonequivariant psd lift of the hyperboloid}\\
\label{ex:hyperboloid}Let $H$ be the hyperboloid in $\RR^3$ defined by:
\[ H = \{ (x_1,x_2,x_3) \in \RR^3 \; : \; x_1,x_2,x_3 \geq 0\text{ and } x_1 x_2 x_3 \geq 1\}. \]
One can construct a psd lift of $H$ of size 6 as follows (see e.g., \cite[page 261]{frgbook}):
\begin{equation}
\label{eq:Hlift}
\begin{aligned}
H &= \left\{ (x_1,x_2,x_3) \in \RR^3 \; : \; \exists y,z \geq 0  \quad
x_1x_2 \geq y^2, \; x_3 \geq z^2, \; yz \geq 1 \right\}\\
&= \left\{ (x_1,x_2,x_3) \in \RR^3 \; : \; \exists y,z  \qquad
\begin{bmatrix} x_1 & y\\ y & x_2 \end{bmatrix} \succeq 0, \;
\begin{bmatrix} x_3 & z\\ z & 1 \end{bmatrix} \succeq 0, \;
\begin{bmatrix} y & 1\\ 1 & z \end{bmatrix} \succeq 0
\right\}.
\end{aligned}
\end{equation}
The hyperboloid $H$ is clearly invariant under permutation of coordinates, i.e., for any permutation $\sigma \in \fS_3$ we have $(x_1,x_2,x_3) \in H \Rightarrow \sigma\cdot (x_1,x_2,x_3) \in H$. However the lift we just constructed does not respect this symmetry: indeed to construct the lift we imposed a particular ordering of the variables where the last coordinate $x_3$ does not play the same role as the first two coordinates $x_1$ and $x_2$.  It is not difficult to formally show that the lift \eqref{eq:Hlift} does not satisfy Definition \ref{def:equivariantpsdlift} of equivariance when $G=\fS_3$. Note however that the lift is equivariant with respect to permuting the coordinates $x_1$ and $x_2$. $\lozenge$
\end{example}

\paragraph{Relation with symmetric LP lift} It is not hard to see that symmetric LP lifts can be interpreted as equivariant psd lifts, where each $\rho(g)$ consists of a permutation matrix.
In fact, recall that an LP lift of a polytope $P$ takes the form $P = \pi(\RR^d_+ \cap L)$ where $L$ is an affine subspace of $\RR^d$ and $\pi:\RR^d \rightarrow \RR^n$ is a linear map. An LP lift $P = \pi(\RR^d_+ \cap L)$ is called $G$-symmetric (or $G$-equivariant) if there exists $\theta:G\rightarrow \fS_d$ (where $\fS_d$ is the group of permutations on $d$ elements) such that for any $y \in \RR^d_+ \cap L$, $\pi(\theta(g) \cdot y) = g\cdot \pi(y)$. By working with diagonal matrices, any symmetric LP lift can be rewritten as an equivariant PSD lift: Indeed, if $P = \pi(\RR^d_+ \cap L)$ is a symmetric LP lift of $P$, then $P = \tilde{\pi}(\S^d_+ \cap \tilde{L})$ is an equivariant psd lift where $\tilde{L} = \{Y \in \S^d, Y \text{ is diagonal and } \diag(Y) \in L\}$ and $\tilde{\pi} = \pi\circ \diag$ where $\diag:\S^d \rightarrow \RR^d$ is the operator extracting the diagonal of a symmetric matrix. This psd lift clearly satisfies the definition of equivariance where $\rho(g)$ is the permutation matrix associated to $\theta(g)$. 

\subsection{Positive semidefinite lifts and sums-of-squares}
\label{sec:intro-sos}

Positive semidefinite lifts have a strong connection with sums-of-squares certificates for nonnegativity.
A generic way to construct positive semidefinite lifts of a polytope $P$ is to look for ``small'' sum-of-squares certificates for the facet-defining inequalities of $P$.
In order to make this statement precise, we first introduce some notations and definitions that will be crucial for the rest of this paper. Given a finite set $X \subset \RR^n$, we let $\cF(X,\RR)$ be the space of real-valued functions on $X$. Note that $\dim \cF(X,\RR) = |X|$. Furthermore, we endow $\cF(X,\RR)$ with pointwise multiplication operation:
\[ (f_1 f_2)(x) = f_1(x) f_2(x) \quad \forall x \in X \quad \forall f_1,f_2 \in \cF(X,\RR). \]
We introduce the following important definition which will be used throughout the paper:
\begin{definition}
Let $X \subset \RR^n$ be a finite set. Let $f\in \cF(X,\RR)$ and let $V$ be a subspace of $\cF(X,\RR)$. We say that $f$ is \emph{$V$-sos on $X$} if there exist functions $f_1,\dots,f_J \in V$ such that $f = f_1^2 + \dots + f_J^2$.
\end{definition}

Given a finite set $X \subset \RR^n$ and $\ell$ a linear form on $\RR^n$, define $\ell_{\max} := \max_{x \in X} \ell(x)$. Observe that $\ell_{\max} - \ell$ takes only nonnegative values on $X$, i.e., $\ell_{\max} - \ell(x) \geq 0$ for all $x \in X$. In other words $\ell_{\max} - \ell|_X \in \cF(X,\RR)$ is a nonnegative function on $X$. The following theorem establishes a relation between sum-of-squares certificates of $\ell_{\max} - \ell$ on $X$ and psd lifts of $\conv(X)$. The theorem is essentially due to Lasserre \cite{lasserre2009convex} and Gouveia et al.  \cite{gouveia2010theta}.
\begin{theoreml}
\label{thm:intro-soslifts}
Let $X \subset \RR^n$ be a finite set and assume that $X$ spans $\RR^n$.
Assume there is a subspace $V$ of $\cF(X,\RR)$ such that for any linear form $\ell$ on $\RR^n$, $\ell_{\max} - \ell$ is $V$-sos on $X$, where $\ell_{\max} := \max_{x \in X} \ell(x)$. Then $\conv(X)$ has an (explicit) psd lift of size $\dim V$.
\end{theoreml}
\begin{remark}
Lasserre \cite{lasserre2009convex} and Gouveia et al. \cite{gouveia2010theta} considered the special case where $V$ is the space of polynomials of degree at most $k$ on $X$. However the same arguments in \cite{lasserre2009convex} and \cite{gouveia2010theta} actually apply more generally to any subspace $V$ of functions in $\cF(X,\RR)$.
\end{remark}

Theorem \ref{thm:intro-soslifts} can actually be (almost) turned into an exact characterization of psd lifts of $\conv(X)$. Indeed one can show that if $\conv(X)$ has a psd lift of size $d$, then there exists a subspace $V$ of $\cF(X,\RR)$ with $\dim V \leq d^2$ such that for any linear form $\ell$ on $\RR^n$, $\ell_{\max} - \ell$ is $V$-sos on $X$. This partial converse to Theorem \ref{thm:intro-soslifts} can be proven using the \emph{factorization theorem} \cite[Theorem 1]{gouveia2011lifts} (see also Theorem \ref{thm:generalpsdliftfactorization} in Appendix \ref{sec:proof_factorization_theorem_equivariant_lifts}).

Assume now that $X$ is invariant under the action of a group $G$ and that we are looking for a $G$-\emph{equivariant} psd lift. In this case it is not hard to formulate a version of Theorem \ref{thm:intro-soslifts} where the psd lift produced is guaranteed to be equivariant. To make this precise, we need the following definition: Observe that the action of $G$ on $X$ induces an action on $\cF(X,\RR)$ defined by: 
\[ (g\cdot f)(x) = f(g^{-1} \cdot x) \]
for any $g \in G, f \in \cF(X,\RR), x \in X$ (this action is known as the left regular representation of $G$). A subspace $V$ of $\cF(X,\RR)$ is called \emph{$G$-invariant} if $g\cdot f \in V$ for any $f \in V$ and $g \in G$. Then one can show that if the subspace $V$ in Theorem \ref{thm:intro-soslifts} is $G$-\emph{invariant}, then the resulting psd lift of $\conv(X)$ is $G$-equivariant:
\begin{theoreml}
\label{thm:intro-equivariantsoslifts}
Let $X \subset \RR^n$ be a finite set that spans $\RR^n$ and assume that $X$ is invariant under the action of a group $G$. Assume furthermore that there is a $G$-invariant subspace $V$ of $\cF(X,\RR)$ such that for any linear form $\ell$ on $\RR^n$, $\ell_{\max} - \ell$ is $V$-sos on $X$, where $\ell_{\max} := \max_{x \in X} \ell(x)$. Then $\conv(X)$ has an (explicit) $G$-equivariant psd lift of size $\dim V$.
\end{theoreml}
\begin{proof}
We provide the proof in Appendix \ref{sec:equivariancesoslift}.
\if\siam1 \qquad \fi
\end{proof}

We saw earlier that Theorem \ref{thm:intro-soslifts} for general psd lifts is actually almost tight, since it admits a partial converse. One of the main contributions of this paper is to show that a similar partial converse to Theorem \ref{thm:intro-equivariantsoslifts} also holds (cf. Theorem \ref{thm:main-structure-general} to follow).

As mentioned earlier, the well-known Lasserre/theta-body lift of $\conv(X)$ can be understood in the context above by taking $V$ to be the space of polynomials of degree at most $k$ on $X$. We say that a function $f \in \cF(X,\RR)$ is a \emph{polynomial of degree at most $k$ on $X$} if there exists a polynomial $p \in \RR[x_1,\dots,x_n]$ with $\deg p \leq k$ such that $f(x) = p(x)$ for all $x \in X$. We denote by $\Pol_{\leq k}(X)$ the subspace of $\cF(X,\RR)$ consisting of polynomials of degree at most $k$ on $X$. The following definition of \emph{theta-rank} from \cite{gouveia2010theta} will be useful later in the paper:
\begin{definition}[Theta-rank \cite{gouveia2010theta}]
\label{def:thetarank}
Let $X$ be a finite subset of $\RR^n$. The \emph{theta-rank} of $X$ is the smallest integer $k$ such that the following holds: for any linear form $\ell$ on $\RR^n$, $\ell_{\max} - \ell$ is $\Pol_{\leq k}(X)$-sos on $X$ where $\ell_{\max} := \max_{x \in X} \ell(x)$.
\end{definition}
Note that if $X$ has theta-rank $k$ then by Theorem \ref{thm:intro-equivariantsoslifts}, $\conv(X)$ has an equivariant psd lift of size $\dim \Pol_{\leq k}(X)$, since $\Pol_{\leq k}(X)$ is a $G$-invariant subspace of $\cF(X,\RR)$.

\paragraph{The explicit lifts} In this paragraph we describe the explicit psd lift that one obtains from Theorems \ref{thm:intro-soslifts} and \ref{thm:intro-equivariantsoslifts}. The details of this construction are not, strictly speaking, needed for the rest of this paper. However we believe that the description given here is beneficial for some of the discussions in Sections \ref{sec:parity} and \ref{sec:cut} on the parity polytope and the cut polytope.
The construction we describe here is essentially due to \cite{lasserre2009convex,gouveia2010theta}. Let $e_i \in \cF(X,\RR)$, for $i=1,\dots,n$ be defined by $e_i(x) = x_i$ and let $\cF(X,\RR)^*$ be the space of linear forms on $\cF(X,\RR)$. If $V$ is a subspace of $\cF(X,\RR)$ that satisfies the assumption of Theorem \ref{thm:intro-soslifts}, then one can show that we have the following explicit psd lift of $\conv(X)$ (Theorem \ref{thm:intro-soslifts}):
\begin{equation}
\label{eq:sosliftV}
\conv(X) = \TH_V(X)
\end{equation}
where
\begin{equation}
\label{eq:SOSLV}
\begin{aligned}
\TH_V(X) := \Biggl\{ (E(e_1),\dots,E(e_n)) \; : \; & E \in \cF(X,\RR)^* \text{ where }\\
 & E(1) = 1, \; E(f^2) \geq 0 \; \forall f \in V \Bigr\}.
\end{aligned}
\end{equation}
Note that by picking a basis for $V$, the constraint $E(f^2) \geq 0 \; \forall f \in V$ in \eqref{eq:SOSLV} can be written as a psd constraint of size $d$ (we refer to Appendix \ref{sec:equivariancesoslift} for more details on how to write \eqref{eq:SOSLV} in the form $\conv(X) = \pi(\S^d_+ \cap L)$). In the definition of $\TH_V(X)$, one can think of $E \in \cF(X,\RR)^*$ as playing the role of an \emph{expectation} operator for a certain measure supported on $X$ ($E$ is sometimes called a \emph{pseudo-expectation}, see e.g., \cite{lee2014lower}). In fact one can compare \eqref{eq:SOSLV} to the following trivial representation of $\conv(X)$:
\begin{equation}
\label{eq:convXexpectation}
\begin{aligned}
\conv(X) = \Biggl\{ (E(e_1),\dots,E(e_n)) \; : \;  & E \in \cF(X,\RR)^* \text{ where } E(f) = \int_{X} f(x) d\mu(x) \\
&   \text{ for some prob. measure $\mu$ on $X$} \Biggr\}.
\end{aligned}
\end{equation}
By comparing Equations \eqref{eq:SOSLV} and \eqref{eq:convXexpectation}, one can easily see that $\conv(X) \subseteq \TH_V(X)$ for any subspace $V$ of $\cF(X,\RR)$.

  Theorem \ref{thm:intro-equivariantsoslifts} asserts that if $V$ is a $G$-invariant subspace, then the psd lift \eqref{eq:sosliftV}-\eqref{eq:SOSLV} is $G$-equivariant.
The Lasserre/theta-body hierarchy of $\conv(X)$ can now be defined as follows:
\begin{definition}[Lasserre/theta-body hierarchy]
Let $X$ be a finite subset of $\RR^n$ and let $\cF(X,\RR)$ be the space of functions on $X$. The \emph{$k$'th level of the Lasserre/theta-body hierarchy} of $X$ is the set $\TH_V(X)$ where $V = \Pol_{\leq k}(X)$ is the space of polynomials of degree at most $k$ on $X$. This set will be denoted $\TH_k(X)$ for simplicity.
\end{definition}
In light of this definition note that the theta-rank of $X$ is the smallest $k$ such that the equality $\conv(X) = \TH_k(X)$ holds.
\begin{remark}
In \cite{gouveia2010theta}, the theta-body hierarchy and the theta-rank are defined for an ideal $I$ of $\RR[x]$. Our definition for the theta-rank of $X \subset \RR^n$ corresponds to the theta-rank of the vanishing ideal of $X$.
\end{remark}

\subsection{Contributions}
\label{sec:contribution}

\subsubsection{Summary}
In this paper we present a representation-theoretic framework to study equivariant psd lifts of a general class of symmetric polytopes known as \emph{orbitopes} \cite{sanyal2011orbitopes}.
\begin{itemize}
\item Our first contribution is a \emph{structure theorem} which can be regarded as a converse to Theorem \ref{thm:intro-equivariantsoslifts}: namely we show that any equivariant psd lift of size $d$ of an orbitope must have the form given in Theorem \ref{thm:intro-equivariantsoslifts} where the dimension of $V$ is bounded by $d^3$. Furthermore, the bound $d^3$ can be improved when the symmetry group has some specific structure. One consequence of the structure theorem is that in order to study equivariant psd lifts, one has to understand the structure of low-dimensional invariant subspaces in $\cF(X,\RR)$.
In particular if one can show that low-dimensional invariant subspaces are composed only of low-degree polynomials, then any lower bound on the Lasserre/theta-body hierarchy would automatically translate to a lower bound on equivariant psd lifts.
\item We apply the theorem to two well-known examples where such a phenomenon occurs, namely the parity polytope and the cut polytope. In both cases we show that low-dimensional invariant subspaces of $\cF(X,\RR)$ correspond (essentially) to low-degree polynomials. Thus using the structure theorem, this shows that if we have a small equivariant psd lift, then the sum-of-squares hierarchy is exact after a small number of steps. For the parity polytope we prove a lower bound of $\lceil n/4 \rceil$ on the sum-of-squares hierarchy which implies an exponential lower bound on the size of any equivariant psd lift of the parity polytope. Similarly for the cut polytope using the well-known lower bound of Laurent of $\lceil n/2 \rceil$ \cite{laurent2003lower} our results yield an exponential lower bound on the size of any equivariant psd lift of the cut polytope.
\end{itemize}

\subsubsection{Statement of results}
\label{sec:statement}
In this brief section we give a more precise statement of our results. The family of polytopes that we focus on in this paper are known as \emph{orbitopes} \cite{barvinok1988convex,barvinok2005convex,sanyal2011orbitopes}: 
Let $G$ be a finite subgroup of $GL(\RR^n)$, the group of $n\times n$ real invertible matrices.
Given $x_0 \in \RR^n$, let $X = G\cdot x_0 := \{ g\cdot x_0 \; : \; g \in G \}$ be the \emph{orbit} of $x_0$ under $G$, and consider the polytope $P$ defined as the convex hull of $X$:
\begin{equation}
 \label{eq:orbitope}
 P = \conv(X) = \conv(G\cdot x_0).
\end{equation}
Such a polytope is called an orbitope \cite{sanyal2011orbitopes}.

Orbitopes are symmetric by construction; for example they are clearly invariant under the action of $G$. In this paper we are interested in psd lifts of orbitopes that respect their symmetry. We prove that any equivariant psd lift of size $d$ of an orbitope $P$ is essentially a sum-of-squares lift where the functions in the sum-of-squares decomposition come from a certain $G$-invariant subspace of dimension at most $d^3$. More precisely we prove the following theorem:
\begin{theorem}
\label{thm:main-structure-general}
Let $P$ be a $G$-orbitope as in \eqref{eq:orbitope}.
Assume $P$ has a $G$-equivariant psd lift of size $d$. Then there exists a $G$-invariant subspace $V$ of $\cF(X,\RR)$ such that:
\begin{itemize}
\item[(i)] For any linear form $\ell$ on $\RR^n$, $\ell_{\max} - \ell$ is $V$-sos on $X$, where \mbox{$\ell_{\max} := \displaystyle{\max_{x \in X} \ell(x)}$}.
\item[(ii)] $\dim V \leq d^3$.
\end{itemize}
\end{theorem} 
The theorem above establishes a connection between small $G$-equivariant psd lifts and low-dimensional $G$-invariant subspaces of $\cF(X,\RR)$. Thus to obtain lower bounds on the sizes of equivariant psd lifts one has to study the structure of such subspaces of $\cF(X,\RR)$. In particular, if one can show that such subspaces correspond to low-degree polynomials, then any lower bound on the sum-of-squares hierarchy will yield a lower bound on $G$-equivariant psd lifts.

In the case where the group $G$ has a certain product structure, one can strengthen the conclusion of Theorem \ref{thm:main-structure-general} with a better bound on the dimension of $V$ (we refer the reader to Section \ref{sec:regular-orbitopes} for the precise statement). Using these structure theorems we study certain well-known orbitopes and we prove lower-bounds on the size of equivariant psd lifts.

\paragraph{Parity polytope} The first example we study is the so-called \emph{parity polytope}, denoted $\PAR_n$, and which is defined as the convex hull of all points $x \in \{-1,1\}^n$ that have an even number of $-1$'s:
\[ \PAR_n = \conv \left\{ x \in \{-1,1\}^n \; : \; \prod_{i=1}^n x_i = 1 \right\}. \] 
It is easy to see that $\PAR_n$ is an orbitope. The symmetry group of $\PAR_n$ consists of coordinate permutations as well as transformations that switch the sign of an even number of coordinates. Let $G_{\parity}$ be the subgroup of $GL(\RR^n)$ generated by these transformations. We show that low-dimensional $G_{\parity}$-invariant subspaces of $\cF(X,\RR)$ consist of low-degree polynomials (where $X$ denotes the vertices of the parity polytope). We also show that the sum-of-squares hierarchy for $\PAR_n$ requires at least $\lceil n/4 \rceil$ levels to be exact. Thus, using the structure theorem, this allows us to obtain the following exponential lower bound on the size of any $G_{\parity}$-equivariant psd lift of the parity polytope:

\begin{theorem}
\label{thm:main-parity-lb}
Any $G_{\parity}$-equivariant psd lift of $\PAR_n$ for $n\geq 8$ must have size $\geq \binom{n}{\lceil n/4 \rceil}$.
\end{theorem}

Note that in the theorem above, equivariance is with respect to the full symmetry group $G_{\parity}$ consisting of coordinate permutations \emph{and} even sign switches. As we see in Section \ref{sec:parity} there are well-known \emph{polynomial-size} LP lifts of $\PAR_n$, however they are \emph{not} equivariant with respect to the full symmetry group.

\paragraph{Cut polytope} Using the framework described above we also study the \emph{cut polytope} defined as:
\[ \CUT_n = \conv\left\{xx^T \; : \; x \in \{-1,1\}^n\right\} \subset \S^n. \]
Let $G_{\cube} \subset GL(\RR^n)$ be the symmetry group of the hypercube $[-1,1]^n$ which consists of signed permutation matrices, i.e., permutation matrices where nonzero entries are either $+1$ or $-1$ (the group $G_{\cube}$ is also known as the \emph{hyperoctahedral group}). The cut polytope is invariant under the action of  $G_{\cube}$ which acts on the space of $n\times n$ symmetric matrices by congruence transformations. By studying low-dimensional invariant subspaces of functions on the vertices of the hypercube $\{-1,1\}^n$ and using the lower bound of Laurent \cite{laurent2003lower} on the sum-of-squares hierarchy for the cut polytope, we prove the following exponential lower bound on the size of any equivariant psd lift of $\CUT_n$:
\begin{theorem}
\label{thm:main-cut-lb}
Any $G_{\cube}$-equivariant psd lift of $\CUT_{n}$ must have size $\geq \binom{n}{\lceil n/4 \rceil}$.
\end{theorem}

Note that Theorems \ref{thm:main-parity-lb} and \ref{thm:main-cut-lb} above are stated in terms of \emph{exact} psd lifts of the parity polytope and cut polytope. Our framework also allows us to consider \emph{approximate} lifts. For example, for the parity polytope and the cut polytope, we show that the approximation quality of any \emph{approximate} equivariant psd lift of small size can be achieved with a small number of rounds of the sum-of-squares hierarchy. We refer to Sections \ref{sec:parity} and \ref{sec:cut} for the precise statements of these results.

\paragraph{Related work}
In independent work, Lee et al. \cite{lee2014power} have also considered symmetric psd lifts for the maximum cut problem (and more generally constraint satisfaction problems) and proved that such lifts cannot be much stronger than the sum-of-squares hierarchy. In their work however they adopted a definition of \emph{symmetric psd lift} that is different from the one we consider here: in \cite{lee2014power}, a psd lift of the cut polytope $\CUT_n = \pi(\S^d_+ \cap L)$ is called symmetric if for any permutation $\sigma$ of $\{1,\dots,n\}$ there exists a $d\times d$ \emph{permutation matrix} $\rho(\sigma)$ such that $\pi(\rho(\sigma) Y \rho(\sigma)^T) = \sigma \cdot \pi(Y)$ for all $Y \in \S^d_+ \cap L$ (where $\sigma \cdot \pi(Y)$ denotes the natural action of $\fS_n$ on $\S^n$ which permutes rows/columns). Note that this definition is more restrictive than ours since it requires $\rho(\sigma)$ to be a permutation matrix whereas in our Definition \ref{def:equivariantpsdlift}, we allow $\rho(\sigma)$ to be any invertible matrix in $GL(\RR^d)$. In this regard our framework is more general and applies to a wider class of psd lifts.

Our framework also allows us to study equivariant lifts with respect to arbitrary groups and not only permutation symmetry. Note for example that in Theorem \ref{thm:main-cut-lb} we consider lifts of the cut polytope that are equivariant with respect to $G_{\cube}$ which is larger than $\fS_n$. For this reason the lower bound as stated in Theorem \ref{thm:main-cut-lb} cannot be directly compared with the lower bound in \cite{lee2014power}. By modifying Lemma \ref{lem:cubegroupdecomp} however one could actually prove a result similar to Theorem \ref{thm:main-cut-lb} where the group $G_{\cube}$ is replaced with $\fS_n$: the lower bound we get is slightly worse but is still exponential in $n$.

Finally, note that in a very recent breakthrough paper \cite{lee2014lower} it was shown that any psd lift of the cut polytope must have exponential size (with no equivariance assumption). In fact the authors of \cite{lee2014lower} propose a new technique to obtain lower bounds on the psd rank \cite{psdranksurvey} which they apply to get exponential lower bounds for the cut polytope, the traveling salesman polytope, as well as the stable set polytope.

\paragraph{Organization of the paper} The paper is organized as follows. In Section \ref{sec:background} we review some background material in representation theory which are used later in the paper. In Section \ref{sec:orbitopes} we describe the general setting of the paper and we prove the structure theorem (Theorem \ref{thm:main-structure-general}) and we also consider the special case of symmetry groups with a product structure. In Section \ref{sec:parity} we study the parity polytope. We first prove a lower bound of $\lceil n/4 \rceil$ on the sum-of-squares hierarchy of the parity polytope. We then prove a key lemma showing that in $\cF(X,\RR)$ (where $X$ is the set of vertices of the parity polytope), any low-dimensional invariant subspace can be realized by low-degree polynomials. This allows us to obtain an exponential lower bound on the size of any equivariant psd lift of the parity polytope. Finally in Section \ref{sec:cut} we consider the cut polytope. Just as for the parity polytope we study low-dimensional invariant subspaces of $\cF(\{-1,1\}^n,\RR)$ and we show that they (essentially) consist of low-degree polynomials. This lemma combined with the structure theorem and the lower bound of Laurent \cite{laurent2003lower} yields an exponential lower bound on the size of any equivariant psd lift of the cut polytope.

\paragraph{Notations} We collect here some of the notations that are used throughout the paper. We denote by $\S^n$ the space of $n\times n$ real symmetric matrices and by $\S^n_+$ the cone of positive semidefinite matrices. The group of $n\times n$ real invertible matrices is denoted by $GL(\RR^n)$. The group of $n\times n$ orthogonal matrices is denoted $O(\RR^n)$. Also if $V$ is a vector space we let $V^*$ be the dual space of $V$ which consists of linear forms on $V$. For a set $X$ we let $\cF(X,\RR)$ be the space of real-valued functions on $X$.
The symmetric group on $n$ elements is denoted by $\fS_n$.
For a matrix $A$ with real entries we denote by $\|A\|_F$ the Frobenius norm of $A$ defined by $\|A\|_F = \trace(A^T A) =  (\sum_{i,j} A_{ij}^2)^{1/2}$.


\section{Background in representation theory}
\label{sec:background}

In this section we recall some basic facts concerning representation theory of finite groups which will be used later in the paper. We refer to \cite{serre1977linear} for a reference. Given a finite group $G$, a real finite-dimensional representation of $G$ is a pair $(V,\rho)$ where $V$ is a real finite-dimensional vector space and $\rho:G \rightarrow GL(V)$ is a group homomorphism. Two representations $(V_1,\rho_1)$ and $(V_2,\rho_2)$ are called $G$-isomorphic if there is an isomorphism $f:V_1\rightarrow V_2$ such that $f(\rho_1(g) x) = \rho_2(g) f(x)$ for all $x \in V_1$ and $g\in G$. A subspace $W$ of $V$ is an \emph{invariant subspace} for the representation $\rho$ if for any $x \in W$ and $g \in G$ we have $\rho(g) x \in W$.  The representation $(V,\rho)$ of $G$ is called \emph{irreducible} if it does not contain any invariant subspace except $\{0\}$ and $V$ itself.
Irreducible representations of a group $G$ are the building blocks of any representation of $G$. 
The following result is a standard fact in representation theory: any finite-dimensional real representation $(V,\rho)$ of $G$ can be decomposed as
\begin{equation}
 \label{eq:isotypicdecomp}
 V = V_1 \oplus \dots \oplus V_k
\end{equation}
where each $V_i$ is isomorphic to a direct sum of $m_i$ copies of an irreducible representation $W_i$ of $G$. This decomposition \eqref{eq:isotypicdecomp} is a canonical decomposition and is called the \emph{isotypic decomposition} of $V$. It satisfies the following important property: if $W$ is an irreducible subspace of $V$ that is $G$-isomorphic to $W_i$ then $W$ is contained in $V_i$. The subspace $V_i$ is called the isotypic component of the irreducible representation $W_i$ in $V$.

 This decomposition result can be used to prove the following proposition which will be needed later:
\begin{proposition}
\label{prop:lowdiminvsubspace}
Let $(V,\rho)$ be a real finite-dimensional representation of a finite group $G$ and assume
\[ V = W_1 \oplus \dots \oplus W_h \]
is a decomposition of $V$ into irreducibles, i.e., each $W_i$ is an irreducible subspace of $V$. Assume furthermore that the $W_i$ are sorted in nondecreasing order of dimension, i.e., $\dim W_1 \leq \dim W_2 \leq \dots \leq \dim W_h$. Assume $W$ is an invariant subspace of $V$ with $\dim W < \dim W_{i_0}$ for some $i_0 \in \{1,\dots,h\}$. Then necessarily $W$ is contained in the direct sum $W_1 \oplus \dots \oplus W_{i_0-1}$.
\end{proposition}
\begin{proof}
Any irreducible subrepresentation of $W$ is isomorphic to one of the $W_i$ for some $i < i_0$. Thus $W$ is contained in the direct sum of isotypic components of the $W_i$'s for $i < i_0$, thus $W$ is contained in $\bigoplus_{i=1}^{i_0-1} W_i$.
\if\siam1 \qquad \fi
\end{proof}

The following well-known proposition will be also be useful later.
\begin{proposition}
\label{prop:changebasis-orthrep}
Let $\rho:G\rightarrow GL(\RR^n)$ be a real finite-dimensional representation of a finite group $G$. Then there exists an invertible matrix $Q$ such that $Q\rho(g)Q^{-1}$ is orthogonal for all $g \in G$.
\end{proposition}
\begin{proof}
With the choice $Q = (\sum_{g \in G} \rho(g) \rho(g)^T)^{-1/2}$, one can easily verify that $Q\rho(g)Q^{-1}$ is orthogonal for all $g \in G$.
\end{proof}


\section{The structure theorem}
\label{sec:orbitopes}

In this section we state and prove our main structure theorems. We first recall an important result from \cite{gouveia2011lifts} which charaterizes positive semidefinite lifts in terms of certificates of nonnegativity for linear inequalities.  Using this characterization, we prove our general structure theorem, Theorem \ref{thm:main-structure-general}. Then we look at an important special case where the symmetry group has a certain product structure and where the result can be strengthened. This special case will be useful in subsequent sections when we look at the parity polytope and the cut polytope.
We conclude the section with an illustration of the structure theorem on the polytope $P = [-1,1]^2$.

\subsection{The factorization theorem}

An important tool that allows us to study equivariant psd lifts is a certain \emph{factorization theorem}. In \cite{yannakakis1991expressing}, Yannakakis showed that LP lifts of a polytope $P$ can be characterized by nonnegative factorizations of the so-called \emph{slack matrix} of $P$.
For symmetric lifts over general cones, a similar factorization theorem exists and was proved by Gouveia et al. in \cite{gouveia2011lifts}.
The following result, which we will use later in the proof of our main theorem, can be proved using arguments identical to the proof of Theorem 2 in \cite{gouveia2011lifts}. We include a proof in Appendix \ref{sec:proof_factorization_theorem_equivariant_lifts} for completeness.
\begin{theorem}
\label{thm:factorization-equivlift}
Let $G$ be a finite group acting on $\RR^n$ and let $X = G\cdot x_0$ where $x_0 \in \RR^n$.
 Assume $\conv(X) = \pi(\S^d_+ \cap L)$ is a $G$-equivariant psd lift of $\conv(X)$ of size $d$, i.e., there is a homomorphism $\rho:G\rightarrow GL(\RR^d)$ such that conditions (i) and (ii) of Definition \ref{def:equivariantpsdlift} hold. Then there exists a map $A:X \rightarrow \S^d_+$ with the following properties:
\begin{itemize}
\item[(i)] For any linear form $\ell$ on $\RR^n$ there exists $B(\ell) \in \S^d_+$ such that if we let $\ell_{\max} := \max_{x \in X} \ell(x)$ we have:
\[ \ell_{\max} - \ell(x) = \langle A(x), B(\ell) \rangle \quad \forall x \in X. \]
\item[(ii)] The map $A$ satisfies the following equivariance relation:
\[ A(g\cdot x) = \rho(g)A(x)\rho(g)^T \quad \forall x \in X, \; \forall g \in G. \]
In particular if $H$ denotes the stabilizer of $x_0$, then we have:
\begin{equation}
\label{eq:Ax0stabilizer}
 A(x_0) = \rho(h) A(x_0) \rho(h)^T \quad \forall h \in H.
\end{equation}
\end{itemize}
Furthermore, the representation $\rho:G\rightarrow GL(\RR^d)$ can be taken to be orthogonal, i.e., $\rho(g) \in O(\RR^d)$ for all $g \in G$.
\end{theorem}
\begin{proof}
See Appendix \ref{sec:proof_factorization_theorem_equivariant_lifts} for a proof.
\if\siam1 \qquad \fi
\end{proof}

\subsection{Structure theorem for general orbitopes}

We are now ready to state and prove our main structure theorem for orbitopes (Theorem \ref{thm:main-structure-general}) which we restate here for convenience.

\begin{reptheorem}{thm:main-structure-general}
Let $G$ be a finite group acting on $\RR^n$ and let $X = G\cdot x_0$ where $x_0 \in \RR^n$.
Assume $\conv(X)$ has a $G$-equivariant psd lift of size $d$. Then there exists a $G$-invariant subspace $V$ of $\cF(X,\RR)$ such that:
\begin{itemize}
\item[(i)] For any linear form $\ell$ on $\RR^n$, $\ell_{\max} - \ell$ is $V$-sos on $X$, where $\ell_{\max} := \displaystyle{\max_{x \in X} \ell(x)}$.
\item[(ii)] $\dim V \leq d^3$.
\end{itemize}
\end{reptheorem}
\begin{proof}
Assume we have a $G$-equivariant psd lift of size $d$ of $\conv(X)$. By the factorization theorem (Theorem \ref{thm:factorization-equivlift}), we know that there exist maps $A:X\rightarrow \S^d_+$ and $B:(\RR^n)^* \rightarrow \S^d_+$ such that for any linear form $\ell \in (\RR^n)^*$ we have:
\begin{equation}
 \label{eq:factorization_orbitope}
 \ell_{\max} - \ell(x) = \langle A(x), B(\ell) \rangle \quad \forall x \in X
\end{equation}
where $\ell_{\max} := \max_{x \in X} \ell(x)$.
Furthermore, since the lift is $G$-equivariant, the map $A$ satisfies the equivariance relation
\begin{equation}
 \label{eq:equivarianceA}
 A(g\cdot x) = \rho(g)A(x)\rho(g)^T \quad \forall x \in X \;\; \forall g \in G
\end{equation}
where $\rho:G\rightarrow O(\RR^d)$ is a group homomorphism.

Let $H$ be the stabilizer of $x_0$, i.e., $H = \{g \in G : g\cdot x_0 = x_0\}$. Note that the set $X$ can be identified with $G/H$, the set of left cosets of $H$. Furthermore the left action of $G$ on $X$ is isomorphic to the left action of $G$ on $G/H$.
For simplicity of notation, we will thus think of functions on $X$ as functions on $G/H$, or equivalently, as functions on $G$ that are constant on the left cosets of $H$. For example since the point $x_0$ is identified with the left coset $1_G H$ of $H$, we will write $A(1_G)$ instead of $A(x_0)$.

We first show how to construct the subspace $V$ of $\cF(X,\RR) \isom \cF(G/H,\RR)$ and then we prove that it satisfies the properties of the statement.

$\bullet$ \emph{Definition of $V$}: Let
\[
A(1_G) = \sum_{i=1}^r \lambda_i P_{W_i}
\]
be an eigendecomposition of $A(1_G)$, where each $P_{W_i}$ is an orthogonal projection on the eigenspace $W_i$. Observe that from Equation \eqref{eq:Ax0stabilizer} we have $A(1_G) = \rho(h) A(1_G) \rho(h)^T$ for all $h \in H$. Since $\rho(h)$ is orthogonal, this means that $A(1_G)$ commutes with $\rho(h)$, and so $\rho(h) W_{i} = W_i$ for each eigenspace $W_i$ of $A(1_G)$, which also implies that $\rho(h) P_{W_i} \rho(h)^T = P_{W_i}$. An important consequence of this is that the functions $g\mapsto \rho(g) P_{W_i} \rho(g)^T$ are constant on the left cosets of $H$, thus we can think of them as functions on $G/H$. Let $V$ be the subspace of $\cF(G/H,\RR)$ spanned by the matrix entries of $x \in G/H\mapsto \rho(x) P_{W_i} \rho(x)^T$, namely
\begin{equation}
\label{eq:defVgeneral}
V = \linspan \Bigl\{ x\in G/H\mapsto (\rho(x) P_{W_i} \rho(x)^T)_{k,l}, \; i=1,\dots,r \;  \text{ and } \; k,l=1,\dots,d \Bigr\}.
\end{equation}

$\bullet$ It is clear that $V$ is $G$-invariant (since $\rho$ is a homomorphism) and that $\dim V \leq d^3$. It thus remains to show that $\ell_{\max} - \ell$ is $V$-sos for any $\ell \in (\RR^n)^*$. We know from \eqref{eq:factorization_orbitope} that there exists $B=B(\ell) \in \S^d_+$ such that $\ell_{\max} - \ell(x) = \langle A(x), B \rangle$ for any $x \in X$. Now for any $x \in G/H$ we have:
\begin{equation}
\label{eq:structurethm-mainlemma-eqAxB}
\langle A(x), B \rangle = \langle \rho(x) A(1_G) \rho(x)^T, B \rangle = \sum_{i=1}^r \lambda_i \langle \rho(x) P_{W_i} \rho(x)^T, B \rangle.
\end{equation}
Note that even though $\rho(x)$ is \emph{not} well-defined when $x \in G/H$, the quantities $\rho(x) A(1_G) \rho(x)^T$ and $\rho(x) P_{W_i} \rho(x)^T$ are well-defined and do not depend on the representative of the coset $x \in G/H$ (cf. previous remarks).
Observe that since $P_{W_i}$ is an orthogonal projection and $\rho(x)$ is orthogonal, we have $(\rho(x) P_{W_i} \rho(x)^T)^2 = \rho(x) P_{W_i} \rho(x)^T$. Thus continuing from Equation \eqref{eq:structurethm-mainlemma-eqAxB} we can write:
\[
\langle A(x), B \rangle = \sum_{i=1}^r \lambda_i \langle (\rho(x) P_{W_i} \rho(x)^T)^2, B \rangle
=
\sum_{i=1}^r \lambda_i \|L^T \rho(x) P_{W_i} \rho(x)^T\|_F^2
\]
where $L$ is such that $B = L^T L$. Since each entry function of 
\[ x\in G/H \mapsto L^T \rho(x) P_{W_i} \rho(x)^T \]
lives in $V$, this shows that $\langle A(x), B \rangle$ is a sum-of-squares of functions in $V$, which is what we wanted.
\end{proof}

\subsection{Special case: regular orbitopes}
\label{sec:regular-orbitopes}

In this section we look at a special case where the symmetry group $G$ has a certain product structure. In this case we show that one can improve Theorem \ref{thm:main-structure-general} and obtain a better bound on the dimension of the subspace $V$.

We saw in the proof of Theorem \ref{thm:main-structure-general} that the stabilizer $H$ of $x_0$ plays an important role. In this section we consider the situation where the group $G$ has a \emph{semidirect product} structure
\begin{equation}
\label{eq:Gsemidirectproduct}
G = N\rtimes H
\end{equation}
where $N$ is a normal subgroup of $G$ and $H$ is the stabilizer of $x_0$. Concretely, \eqref{eq:Gsemidirectproduct} simply means that any element $g \in G$ can be written in a unique way as $g = nh$ where $n \in N$ and $h \in H$.
The following example illustrates this situation.
\begin{example}[The hypercube]
\label{ex:hypercube-regular-orbitope}
Let $G \subset GL(\RR^n)$ be the group of signed permutations, i.e., the group of permutation matrices where nonzero entries are either $+1$ or $-1$. Let $x_0 = (1,\dots,1)$, and note that $G\cdot x_0 = \{-1,1\}^n$. The stabilizer of $x_0$ is $\fS_n$, the group of permutation matrices (with positive signs everywhere). Observe that $G$ has the following semidirect product structure: $G = N \rtimes \fS_n$ where $N$ is the subgroup of diagonal matrices with $+1$ or $-1$ on the diagonal. $\lozenge$
\end{example}

Let $X = G\cdot x_0$. If \eqref{eq:Gsemidirectproduct} holds and $H$ is the stabilizer of $x_0$, then it is easy to see that we have $G\cdot x_0 = N\cdot x_0$, i.e., $X$ is an $N$-orbitope: indeed if $x = g\cdot x_0$ for some $g \in G$, then since we can write $g = nh$ where $n \in N$ and $h \in H$ we have $g\cdot x_0 = nh\cdot x_0 = n\cdot x_0$ since $h \cdot x_0 = x_0$. Furthermore, one can show that for any $x \in N\cdot x_0$, there is a \emph{unique} $n \in N$ such that $n\cdot x_0 = x$ (indeed, if $n,n' \in N$ are such that $n\cdot x_0 = n'\cdot x_0$, then $n'^{-1} n$ stabilizes $x_0$ and so by the semidirect product property we must have $n'^{-1} n = 1_G$). In this case we say that $X$ is an $N$-\emph{regular orbitope}, since the action of $N$ on $X$ is \emph{regular}.

We now state and prove our structure theorem in this particular case.

\begin{theorem}
\label{thm:main-structure-product}
Let $G$ be a finite group acting on $\RR^n$ and let $X = G\cdot x_0$ where $x_0 \in \RR^n$.
Let $H$ be the stabilizer of $x_0$ and assume that $G$ has the form $G = N\rtimes H$ where $N$ is a normal subgroup of $G$.
Assume $P$ has a $G$-equivariant psd lift of size $d$. Then there exists a $G$-invariant subspace $V$ of $\cF(X,\RR)$ with the following properties:
\begin{itemize}
\item[(i)] For any linear form $\ell$ on $\RR^n$, $\ell_{\max} - \ell$ is $V$-sos on $X$, where $\ell_{\max} := \displaystyle{\max_{x \in X} \ell(x)}$.
\item[(ii)] $\dim V \leq \alpha_N(d) \cdot d$ where $\alpha_N(d)$ is the largest dimension of any real irreducible representation of $N$ of dimension $\leq d$ (in particular $\dim V \leq d^2$).
\end{itemize}
\end{theorem}
\begin{proof}
Assume we have a $G$-equivariant psd lift of size $d$ of $\conv(X)$. By the factorization theorem (Theorem \ref{thm:factorization-equivlift}), we know that there exist maps $A:X\rightarrow \S^d_+$ and $B:(\RR^n)^* \rightarrow \S^d_+$ such that for any linear form $\ell \in (\RR^n)^*$ we have:
\begin{equation}
 \label{eq:factorization_orbitope_product}
 \ell_{\max} - \ell(x) = \langle A(x), B(\ell) \rangle \quad \forall x \in X
\end{equation}
where $\ell_{\max} := \max_{x \in X} \ell(x)$.
Furthermore, since the lift is $G$-equivariant, the map $A$ satisfies the equivariance relation
\begin{equation}
 \label{eq:equivarianceA_product}
 A(g\cdot x) = \rho(g)A(x)\rho(g)^T \quad \forall x \in X \;\; \forall g \in G
\end{equation}
where $\rho:G\rightarrow GL(\RR^d)$ is a group homomorphism (in fact we can choose $\rho$ to take values in $O(\RR^d)$ however we will not need this in this proof).

Note that since $G$ has the product structure $G = N\rtimes H$ we can identify $X$ with $N$, and the left action of $G$ on $X$ is isomorphic to the left action of $G$ on $N$ defined by $g\cdot x = nh x h^{-1} \in N$, where $g=nh \in G$ and $x \in N$.
Thus in the rest of the proof we will think of functions on $X$ as functions on $N$. For example since the point $x_0$ is identified with $1_N$, we will write $A(1_N)$ instead of $A(x_0)$.

We now define the subspace $V$ of $\cF(X,\RR) \isom \cF(N,\RR)$ and then we show it has the required properties. 

$\bullet$ \emph{Definition of $V$}: Let $V$ be the subspace of $\cF(N,\RR)$ spanned by the matrix entry functions of $\rho|_N$, i.e.,
\[
V = \linspan\Bigl\{ x \in N \mapsto \rho(x)_{ij}, i,j=1,\dots,d \Bigr\}.
\]

$\bullet$ We need to show that $V$ is a $G$-invariant subspace and that Properties (i) and (ii) in the statement of the theorem are satisfied.
\begin{itemize}
\item[$\ast$] To see why $V$ is $G$-invariant, note that for any $x \in N$ and $g = nh \in G$ we have $\rho(g\cdot x) = \rho(nhxh^{-1}) = \rho(nh) \rho(x) \rho(h^{-1})$ thus for any $i,j$ the function $x\mapsto \rho(g\cdot x)_{ij}$ is a linear combination of the functions $x\mapsto \rho(x)_{k,l}$. This shows that $V$ is $G$-invariant.
\item[$\ast$] To prove that $\dim V \leq \alpha_N(d) \cdot d$, observe that if $n_1,\dots,n_k$ are the dimensions of the irreducible components of the representation $\rho|_{N}:N\rightarrow GL(\RR^d)$ of $N$, then the matrices $\rho(x)$ ($x \in N$) are all, up to a global change of basis, block-diagonal with blocks of size $n_1,\dots,n_k$. Thus we have $\dim V \leq \sum_i n_i^2 \leq \sum_i n_i \alpha_N(d) = \alpha_N(d) \cdot d$ since each $n_i \leq \alpha_N(d)$ and $\sum_i n_i = d$.
\item[$\ast$] Finally, it remains to prove Property (i). Let $\ell \in (\RR^n)^*$ and let $B=B(\ell) \in \S^d_+$ such that \eqref{eq:factorization_orbitope_product} is true. Note that for any $x \in N$ we have:
\begin{equation}
\label{eq:ABkron1}
\begin{aligned}
\langle A(x), B \rangle &= \langle \rho(x) A(1_N) \rho(x)^T, B \rangle
= \trace(\rho(x) A(1_N) \rho(x)^T B).
\end{aligned}
\end{equation}
Since $A(1_N)$ and $B$ are positive semidefinite matrices, we can write $A(1_N) = L_A L_A^T$ and $B = L_B L_B^T$ for some matrices $L_A$ and $L_B$. Then we have:
\[
\langle A(x), B \rangle = \trace(\rho(x) L_A L_A^T \rho(x)^T L_B L_B^T) = \|L_B^T \rho(x) L_A\|_F^2.
\]
Since each entry function of $x \mapsto L_B^T \rho(x) L_A$ lives in $V$, it follows that $x\mapsto \langle A(x), B \rangle$ has a sum-of-squares representation with functions from $V$.  This completes the proof.
\end{itemize}
\end{proof}

\begin{remark}
\label{rem:approx}
Observe that the theorem above could be stated more generally without using the language of \emph{lifts}. Assume $p(x)$ is a function that is nonnegative on $X$ and has an \emph{equivariant certificate of nonnegativity} of the form
\[ p(x) = \langle A(x), B\rangle \quad \forall x \in X \]
where $B\in \S^d_+$ and $A:X\rightarrow \S^d_+$ satisfies the equivariance relation:
\[ A(g\cdot x) = \rho(g) A(x) \rho(g)^T \quad \forall x \in X, \; \forall g \in G. \]
Then the arguments in the proofs above show that $p(x)$ must be a sum of squares of functions from a $G$-invariant subspace $V$ of $\cF(X,\RR)$ whose dimension is bounded by a certain function of $d$ ($d^3$ in the setting of Theorem \ref{thm:main-structure-general} and $\alpha_N(d) d$ in the setting of Theorem \ref{thm:main-structure-product}). In the theorems above, the function $p(x)$ corresponds to the facet-defining linear function $p(x)=\ell_{\max} - \ell(x)$ but the proofs did not use this specific form of the function $p(x)$. $\lozenge$
\end{remark}

\subsection{Example: the square $[-1,1]^2$}

Before concluding this section we illustrate Theorem \ref{thm:main-structure-general} by considering the example of the square $P=[-1,1]^2$. As we saw in Example \ref{ex:hypercube-regular-orbitope} the square $[-1,1]^2$ is a $N$-regular orbitope where the group $N$ is the group of $2\times 2$ diagonal matrices with $+1$ or $-1$ on the diagonal and where $x_0 = (1,1) \in \RR^2$. The full symmetry group of the square $[-1,1]^2$ is the dihedral group of order 8, which we will denote by $G$ here. Note that $G$ has a semidirect product structure $G = N\rtimes \fS_2$ where $\fS_2$ is the group of $2\times 2$ permutation matrices and corresponds to the stabilizer of $x_0$.

In Example \ref{ex:square} in the introduction we gave a construction of a $G$-equivariant psd lift of $P$ of size 3. 
 We can apply Theorem \ref{thm:main-structure-product} to this lift: Theorem \ref{thm:main-structure-product} says that there must exist a $G$-invariant subspace $V$ of $\cF(\{-1,1\}^2,\RR)$ with the following properties:
\begin{itemize}
\item[(i)] Any facet inequality $\ell(x) \leq \ell_{\max}$ of $P$ has a sum-of-squares certificate with functions from $V$:
\[ \ell_{\max} - \ell(x) = \sum_{j} f_j(x)^2 \quad \forall x \in \{-1,1\}^2 \]
where $f_j \in V$.
\item[(ii)] $\dim V \leq 1\cdot 3  = 3$ (indeed $\alpha_N(3) = 1$ since $N$ is isomorphic to $\ZZ_2^2$ for which all the real irreducible representations have dimension one).
\end{itemize}
It is actually not difficult to construct this subspace $V$ explicitly. Indeed, let $V$ be the subspace spanned by the following three functions:
\[ (x_1,x_2) \in \{-1,1\}^2 \mapsto 1, \quad (x_1,x_2) \in \{-1,1\}^2 \mapsto x_1, \quad (x_1,x_2) \in \{-1,1\}^2 \mapsto x_2. \]
Clearly $\dim V \leq 3$, and it is easy to see that this subspace is $G$-invariant. To check that point (i) above is true consider for example the facet inequality $1-x_1 \geq 0$ of $P$. Then we can verify that for any $x \in \{-1,1\}^2$ we have:
\[ 1-x_1 = f(x)^2 \]
with $f(x) = (1-x_1)/\sqrt{2}$ (note that $f \in V$). The same is also true for the other facet inequalities. Thus the subspace $V$ we just constructed satisfies points (i) and (ii).

Now one may wonder if there exists an equivariant psd lift of the square $P=[-1,1]^2$ of size 2. Using Theorem \ref{thm:main-structure-product} this would mean that there exists a $G$-invariant subspace $V$ of $\cF(\{-1,1\}^2)$ of dimension $\leq 2$ which allows us to certify the four facet inequalities of $P$ using sum-of-squares. Later in the paper we will study in more detail the space of functions on the hypercube $\{-1,1\}^n$ and their invariant subspaces, cf. Lemma \ref{lem:cubegroupdecomp}. Using this lemma one can actually show that such a subspace $V$ of dimension $\leq 2$ cannot exist, ruling out the existence of $G$-equivariant psd lifts of size 2 of the square $[-1,1]^2$.

\begin{remark} Actually it is known that there does not exist any psd lift (even a nonequivariant one) of the square $[-1,1]^2$ of size 2. Indeed it was shown in \cite{gouveia2013polytopes} that any psd lift of a full-dimensional polytope $P$ in $\RR^n$ must have size at least $n+1$. $\lozenge$
\end{remark}


\section{The parity polytope}
\label{sec:parity}

In this section we consider the case of the parity polytope and we prove exponential
lower bounds on equivariant psd lifts using the results from the previous
section.

\subsection{Definitions}

Define $\EVEN_n$ to be the set of points $x \in \{-1,1\}^n$ that have an even
number of $-1$'s, i.e.:
\begin{equation}
\label{eq:evenpoints}
\EVEN_n = \left\{ x \in \{-1,1\}^n \; : \; \prod_{i=1}^n x_i = 1\right\}.
\end{equation}
The convex hull of $\EVEN_n$ is called the \emph{parity polytope} and is
denoted $\PAR_n$:
\[ \PAR_n = \conv(\EVEN_n). \]

\paragraph{Symmetry group} Note that the parity polytope is the orbitope $\conv(N_{\parity} \cdot x_0)$ where $x_0 = \mathbf{1} \in \RR^n$ is the vector of all-ones and $N_{\parity} \subset GL(\RR^n)$ is the subgroup of $GL(\RR^n)$ consisting of
diagonal matrices with $+1$ or $-1$ on the diagonal and with an even number of
$-1$'s, i.e.:
\[ N_{\parity} = \{\diag(x) \; : \; x \in \EVEN_n \}. \]
The parity polytope is invariant under the action of
$N_{\parity}$, i.e., it is invariant under switching an even number of signs.
In addition to $N_{\parity}$-invariance, the parity polytope is also invariant
under permutation of coordinates. Let $G_{\parity}$ be the group that
consists of \emph{evenly-signed permutation matrices}, i.e., signed permutation
matrices with an even number of $-1$'s; then one can check that $\PAR_n$ is
invariant under the action of $G_{\parity}$. Note that the group $G_{\parity}$ consists of all possible products of the form $\epsilon h$
where $\epsilon \in N_{\parity}$ and $h \in \fS_n$:
\[ G_{\parity} := \{ \epsilon h \; : \; \epsilon \in N_{\parity}, h \in \fS_n \} = N_{\parity} \rtimes \fS_n. \]
The group $G_{\parity}$ has the product structure described in Section
\ref{sec:regular-orbitopes}. In fact we have $G_{\parity} = N_{\parity} \rtimes \fS_n$ and $\fS_n$ corresponds to the stabilizer of $x_0 = \mathbf{1} \in \RR^n$. 

\paragraph{Facet description of the parity polytope} When $n > 2$, the parity
polytope is a full-dimensional polytope in $\RR^n$. It has the following
description using linear inequalities:
\begin{equation}
    \label{eq:ppfacet} \PAR_n =
\Bigl\{ x \in \RR^n \; : \;  -1\leq x \leq 1,\;
                       \displaystyle\sum_{i \in A^c} x_i - \displaystyle\sum_{i \in A} x_i \leq n-2 \quad \forall A \subseteq [n], |A| \text{ odd}  \Bigr\}.
\end{equation}
If $n\geq 4$ each of the $2n+2^{n-1}$ inequalities are facet-defining. If $n=3$
the inequalities $-1\leq x \leq 1$ are redundant giving the simpler description
with $4$ facets
\begin{equation}
    \label{eq:ppfacet3}
    \PAR_3 = \{x\in \RR^n \; :\;
\begin{array}[t]{rl}
 x_1+x_2-x_3 & \leq 1\\
 x_1-x_2+x_3 & \leq 1 \\
-x_1+x_2+x_3 & \leq 1\\
-x_1-x_2-x_3 & \leq 1\}.
\end{array}
\end{equation}

\paragraph{Nonequivariant polynomial-size lifts of the parity polytope}
Polynomial-size lifts of the parity polytope have been known since the original
paper of Yannakakis \cite{yannakakis1991expressing}. In fact there are two
known LP lifts of the parity polytope of size respectively $O(n^2)$ and $O(n)$.
The two lifts respect some of the symmetry of the parity polytope however none
of them is equivariant with respect to the \emph{full} symmetry group
$G_{\parity} = N_{\parity} \rtimes \fS_n$.

\begin{itemize}
\item The lift of size $O(n^2)$ given by Yannakakis
\cite{yannakakis1991expressing} relies on the observation that
\[ \PAR_n = \conv\left( \bigcup_{\substack{k \text{ even}}} S_k \right) \]
where $S_k$ are the ``slices'' of the hypercube defined by the equation $\1^T x
= n-2k$:
\[ 
\begin{aligned}
S_k &= \conv\{x \in \{-1,1\}^n \; : \; x \text{ has exactly $k$ components equal to $-1$}\}\\
&= \{ x \in [-1,1]^n \; : \; \1^T x = n-2k \}.
\end{aligned}
 \]
Since each $S_k$ has a simple description using only $O(n)$ linear
inequalities, this can be used to construct a lift of $\PAR_n$ of size $O(n^2)$.
One can easily verify that this lift is equivariant with respect to permutation
of the coordinates. One can show however that this lift is \emph{not}
equivariant with respect to switching an even number of signs.
 Intuitively, the reason behind this is that
the operation of switching signs does not preserve the slices $S_k$.

\item There is a smaller yet less well known LP lift of the parity polytope due
to \cite[Section 2.6.3]{carr2004polyhedral} (see also
\cite{kaibel2011constructing}) which has size $O(n)$. 
This lift is equivariant with respect to switching an even number of signs,
however it  is \emph{not} equivariant with respect to the permutation action.
The key observation behind this LP lift is that $(x_1,x_2,\ldots,x_n)\in
\EVEN_n$ if and 
only if there exists $z\in \{-1,1\}$ such that $(x_1,x_2,z)\in \EVEN_3$ and
$(z,x_3,\ldots,x_n)\in \EVEN_{n-1}$ (simply take $z=x_1x_2$). In fact one can
establish an analog of this statement for the parity polytope:
    \begin{equation}
        \label{eq:ppdp}
\begin{aligned}
     \PAR_n = \{x\in \RR^{n}: \exists z \;\; \text{s.t.} \;\; & (x_1,x_2,z)\in \PAR_3\\
&\text{and} \;\; (z,x_3,\ldots,x_n)\in \PAR_{n-1}\}.
\end{aligned}
\end{equation}
Repeatedly applying \eqref{eq:ppdp} shows that 
\[
\begin{aligned}
 \PAR_n &= \Bigl\{
x\in \RR^n: \exists z_2,z_3,\ldots,z_{n-2}\;\;\text{s.t.}\;\; \\
& \qquad\qquad\qquad (x_1,x_2,z_2)\in \PAR_3,\; (z_{n-2},x_{n-1},x_n)\in \PAR_3,\nonumber\\
    &\qquad\qquad\qquad\text{and}\;\;(z_i,x_{i+1},z_{i+1})\in \PAR_3 \quad\text{for $i\in \{2,3,\ldots,n-3\}$}\Bigr\}.\label{eq:dplift}
\end{aligned}
\]
This description shows that $\PAR_n$ is the projection of a polytope with
$4(n-2)$ facets. It is not too hard to show that this lift is actually
equivariant with respect to switching an even number of signs. However one can
see that this lift is not equivariant with respect to permutations since we
have broken permutation symmetry by imposing a particular ordering on the
variables.
\end{itemize}

\subsection{Functions on $\EVEN_n$}
Let $n\geq 1$. If $I\subseteq [n]$ define the monomial map $\RR^n\ni x\mapsto x^I:= \prod_{i\in I}x_i$. 
We can regard these as functions on $\EVEN_n$ by simply restricting their domain. When we do  
so, we write them as $e_I:\EVEN_n\rightarrow \RR$ defined by $e_I(x) = x^I$. We now describe 
a basis for $\cF(\EVEN_n,\RR)$ in terms of these functions.

\begin{proposition} 
\label{prop:evenbasis}
Let $n \geq 1$.  
\begin{itemize}
\item If $n$ is odd, then the functions $e_I$ for $|I| < n/2$ form a basis of $\cF(\EVEN_n,\RR)$.
\item If $n$ is even, then the functions $e_I$ with $|I| < n/2$ together with the functions $(e_I+e_{I^c})/2$ for $|I| = n/2$ constitute a basis of $\cF(\EVEN_n,\RR)$.
\end{itemize}
\end{proposition}
\begin{proof}
Given $a \in \EVEN_n$, let $1_a\in\cF(\EVEN_n,\RR)$ be the indicator function for
$a$, i.e., $1_a(x) = 1$ if $x=a$ and $1_a(x) = 0$ otherwise.  Clearly the
family of functions $\{1_a\}_{a \in \EVEN_n}$ forms a basis of $\cF(\EVEN_n,\RR)$.
Observe that $1_a$ can be written as the following polynomial:
\[ 1_a(x) = \frac{1}{2^n} (1+a_1 x_1)\dots (1+a_nx_n) \quad \forall x \in \EVEN_n. \]
Furthermore, the following identities are true on $\EVEN_n$: $x_i^2 = 1$ and $x^I = x^{I^c}$
for any $x \in \EVEN_n$. If we expand the polynomial expression for $1_a$ above
using the previous identities we see that, when $n$ is odd, $1_a$ is a linear
combination of the square-free monomials $x^I$ for $|I| < n/2$. When $n$ is
even any monomial of the form $ x^{I}$ where $|I|=n/2$ can be rewritten as 
$(x^{I} + x^{I^c}) / 2$.  Thus this shows that the functions given
in the statement of the proposition form a generating set for $\cF(\EVEN_n,\RR)$.
Since the number of such functions is $2^{n-1} = \dim \cF(\EVEN_n,\RR)$, they form
a basis of $\cF(\EVEN_n,\RR)$.
\if\siam1 \qquad \fi
\end{proof}

Given $0\leq k < n/2$, denote by $\Pol_k(\EVEN_n)$ the subspace of
$\cF(\EVEN_n,\RR)$ spanned by the $e_I$ with $|I|=k$. If $k=n/2$ let $\Pol_k(\EVEN_n)$ be the subspace of $\cF(\EVEN_n,\RR)$ spanned by 
the $(e_I+e_{I^c})/2$ with $|I|=n/2$.
So we have:
\[
\dim \Pol_k(\EVEN_n) = \begin{cases} \binom{n}{k} & \text{ if } k < n/2\\ \frac{1}{2} \binom{n}{n/2} & \text{ if } k = n/2. \end{cases}
\]
The proposition above shows that the space $\cF(\EVEN_n,\RR)$ decomposes as:
\[ \cF(\EVEN_n,\RR) = \Pol_0(\EVEN_n) \oplus \dots \oplus \Pol_{\lfloor n/2\rfloor}(\EVEN_n). \]

The next lemma expresses the pointwise product of the functions $e_I$ and $e_J$ in terms of our basis 
for $\cF(\EVEN_n,\RR)$. (Here, and subsequently,
 if $I,J\subseteq[n]$ then $I\triangle J$ is the symmetric difference of $I$ and $J$.)
\begin{lemma}
	If $I,J\subseteq [n]$ then 
	\[ e_Ie_J = \begin{cases} e_{I\triangle J} & \text{if $|I\triangle J|< n/2$}\\
					e_{(I\triangle J)^c} & \text{if $|I\triangle J| > n/2$}\\
					(e_{I\triangle J} + e_{(I\triangle J)^c})/2 & \text{if $|I\triangle J| = n/2$}.
	\end{cases}\]
\end{lemma}
\begin{proof}
	For any $x\in \EVEN_n$ we have that $x_i^2 = 1$ for all $i\in [n]$ and $\prod_{i\in [n]}x_i = 1$. Hence for any $I,J\subseteq [n]$, 
	$x^Ix^J = x^{I \triangle J} = x^{(I\triangle J)^c} = (x^{I\triangle J}+x^{(I\triangle J)^c})/2$. The result then follows by recognizing
	that at least one of these can be written in terms of the given basis from Proposition~\ref{prop:evenbasis}.
\if\siam1 \qquad \fi
\end{proof}

\subsection{A lower bound on the theta-rank of the parity polytope}

For any $1\leq k \leq n/2$ denote by $\Pol_{\leq k}(\EVEN_n)$ the subspace of $\cF(\EVEN_n,\RR)$ defined by 
\[ \Pol_{\leq k}(\EVEN_n) = \Pol_0(\EVEN_n) \oplus \dots \oplus \Pol_{k}(\EVEN_n).\]
Observe that the theta-rank (cf. Definition~\ref{def:thetarank}) of the parity polytope is the smallest $k$ such that the following holds:
\[
\forall \ell \in (\RR^n)^*, \quad \ell_{\max} - \ell \text{ is $\Pol_{\leq k}(\EVEN_n)$-sos on $\EVEN_n$},
\]
where for a linear form $\ell \in (\RR^n)^*$, we call $\ell_{\max} := \max_{x \in \EVEN_n} \ell(x)$.

Note that if $|I|,|J|<n/4$ then the product of $e_I(x) = x^I$ and $e_J(x) = x^J$ when thought of as functions on $\EVEN_n$ is the same 
as their product when thought of as functions on $\{-1,1\}^n$. 
This is the basic reason why the following lower bound on the theta-rank of the parity polytope holds.
\begin{proposition}
\label{prop:parity-thetarank-lb}
The theta-rank of the parity polytope $\PAR_n$ is at least $\lceil n/4 \rceil$.
\end{proposition}
\begin{proof}
We first sketch the outline of the proof before filling in the details. We
choose a function $\ell_{\max} - \ell\in \Pol_{\leq 1}(\EVEN_n)$ that is
nonnegative on $\EVEN_n$ but when viewed as a polynomial on $\RR^n$ takes a
negative value on some point $p\in \{-1,1\}^n\setminus \EVEN_n$. (One can think
of this as a facet inequality for $\PAR_n$ that is not valid for the
hypercube.) We use this point $p$ to construct a linear functional $\mathcal L_p\in
\cF(\EVEN_n,\RR)^*$ (defined to mimic evaulation of the function at $p$) such
that $\mathcal L_p(\ell_{\max}-\ell) < 0$ and yet whenever $k<n/4$ and $f\in
\Pol_{\leq k}(\EVEN_n)$ we have that $\mathcal L_p(f^2)\geq 0$. This would imply that if $k<n/4$ then  $\mathcal L_p$ 
separates the linear function $\ell_{\max}-\ell$ (that is nonnegative on $\EVEN_n$) from the 
cone of functions that are $\Pol_{\leq k}(\EVEN_n)$-sos on $\EVEN_n$, which would complete the proof. 

We now fill in the details. Let $\ell_{\max} - \ell(x) = (n-2)e_{\emptyset}(x) + e_{\{n\}}(x) - \sum_{i=1}^{n-1}e_{\{i\}}(x)$ and observe that this 
is a nonnegative function on $\EVEN_n$ (since it defines a facet of $\PAR_n$).
Let $p = (1,1,\ldots,1,-1)\in \RR^n$, and note that $p$ has an \emph{odd} number of $-1$s. Define a linear functional $\mathcal L_p\in \cF(\EVEN_n,\RR)^*$ by 
defining it on our basis for $\cF(\EVEN_n,\RR)$ by  
\begin{align*}
	\mathcal L_p(e_I) & = \prod_{i\in I}p_i \;\;\text{if $|I|<n/2$}\;\;\text{and}\\
				\mathcal L_p((e_I+e_{I^c})/2) & = \left(\prod_{i\in I}p_i +\prod_{i\in I^c}p_i\right)/2 = 0 \;\; \text{if $|I|=n/2$}.	
\end{align*}
Observe that $\mathcal L_p(\ell_{\max} - \ell) = (n-2) + \mathcal L_p(e_{\{n\}}) - \sum_{i=1}^{n-1}\mathcal L_p(e_{\{i\}}) = (n-2) - 1 - (n-1) = -2$. 

We now show that if $k<n/4$ and $f\in \Pol_{\leq k}(\EVEN_n)$ then $\mathcal L_p(f^2) \geq 0$. 
Observe that if $|I\triangle J| < n/2$ and $|I|,|J|<n/2$ then 
\begin{equation}
\label{eq:hom}
 \mathcal L_p(e_Ie_J) = \mathcal L_p(e_{I\triangle J}) = \prod_{i\in I\triangle J}p_i = 
	\left(\textstyle{\prod_{i\in I}p_i}\right)\left(\textstyle{\prod_{j\in J}p_j}\right) = \mathcal L_p(e_I)\mathcal L_p(e_J).
\end{equation}
If $f\in \Pol_{\leq k}(\EVEN_n)$ where $k<n/4$ then there are constants $c_I\in \RR$ such that $f = \sum_{m=0}^{k}\sum_{|I|=k}c_Ie_I$.
Hence by~\eqref{eq:hom}, and the observation that if  $|I|,|J|<n/4$ then $|I\triangle J| < n/2$,
we can see that any $f\in \Pol_{\leq k}(\EVEN_n)$ satisfies
\begin{equation}
\begin{aligned}
	\label{eq:ellsq} \mathcal L_p(f^2) &= 
	\sum_{m,m'=0}^{k} \sum_{|I|=m}\sum_{|J|=m'}c_Ic_J\mathcal L_p(e_Ie_J)\\
&= 
	\sum_{m,m'=0}^{k}\sum_{|I|=m}\sum_{|J|=m'}c_Ic_J\mathcal L_p(e_I)\ell_p(e_J) = \mathcal L_p(f)^2\geq 0.
\end{aligned}
\end{equation}
This completes the proof.
%
\if\siam1 \qquad \fi
\end{proof}

\begin{remark}
One can show that the lower bound of $\lceil n/4 \rceil$ on the theta-rank of the parity polytope $\PAR_n$ is in fact tight, cf. Proposition 5 in \cite{polygonsequivariant}.
\end{remark}

\subsection{Equivariant psd lifts of the parity polytope}

In this section we study psd lifts of the parity polytope that are equivariant
under the full symmetry group $G_{\parity}$ of evenly signed permutations. Given two integers $n$ and
$k\leq n/2$ we define:
\begin{equation}
 \label{eq:Dnk}
 D_{n,k} := \begin{cases} \binom{n}{k} & \text{ if $n$ is odd}\\ 
	\min\left(\binom{n}{k},\frac{1}{2} \binom{n}{n/2}\right) & \text{ if $n$ is even.} \end{cases}
\end{equation}
We prove the following theorem:
\begin{theorem}
\label{thm:lb-sym-parity}
Assume $\PAR_n$ has a $\S^d_+$-lift that is $G_{\parity}$-equivariant of
size $d < D_{n,k}$ where $k \leq n/2$. Then the $(k-1)$'st theta-body
relaxation is exact, i.e., $\TH_{k-1}(\EVEN_n) = \PAR_n$.
\end{theorem}
\begin{proof}
Assume we have a $G_{\parity}$-equivariant psd lift of $\PAR_n$ of size $d$. 
We can apply Theorem \ref{thm:main-structure-product} with 
$P = \PAR_n$ and $G = G_{\parity} = N_{\parity}\rtimes \fS_n$. Since 
$N_{\parity} \isom \ZZ_2^{n-1}$, all the real irreducible representations of 
$N_{\parity}$ are one-dimensional. Thus Theorem \ref{thm:main-structure-product} says 
that there exists a $G_{\parity}$-invariant subspace $V$ of $\cF(\EVEN_n,\RR)$ with
 $\dim V \leq d$ such that for any linear form $\ell \in (\RR^n)^*$ we have that
\begin{equation}
 \label{eq:par-sos-certificate}
 \text{$\ell_{\max}-\ell$ is $V$-sos on $\EVEN_n$.}
\end{equation}
In Lemma \ref{lem:parity-invariantsubspaces} (cf.~below) we show that such an
invariant subspace, when $d < D_{n,k}$, is composed entirely of polynomials of
degree at most $k-1$, i.e.\ $V$ is a subspace of $\Pol_{\leq k-1}(\EVEN_n)$. 
Thus this shows that $\TH_{k-1}(\EVEN_n) = \PAR_n$.
\if\siam1 \qquad \fi
\end{proof}

Note that, using Remark \ref{rem:approx} from the previous section, one can
actually state a more general theorem relating \emph{approximate} equivariant psd
lifts and the sum-of-squares hierarchy. Indeed one can prove the following:
\begin{theorem}
\label{thm:approx-parity-sos}
Assume $\widehat{P} \subseteq \RR^n$ is an outer-approximation of $\PAR_n$
(i.e., $\PAR_n \subseteq \widehat{P}$) and assume $\widehat{P}$ has a
$G_{\parity}$-equivariant psd lift of size $d$. If $d < D_{n,k}$ for some
$k \leq n/2$ then necessarily 
\[ \PAR_n \subseteq \TH_{k-1}(\EVEN_n) \subseteq \widehat{P} \] where $\TH_{k-1}(\EVEN_n)$ is the $(k-1)$'st Lasserre/theta-body relaxation for the parity polytope.
\end{theorem}
\begin{proof} 
The proof is very similar to the proof of Theorem \ref{thm:lb-sym-parity}
above.  Given a linear form $\ell \in (\RR^n)^*$ let $\ell_{\max}$ and
$\widehat{\ell_{\max}}$ be respectively the maximum of $\ell$ on $\PAR_n$ and
$\widehat{P}$. Note that $\ell_{\max} \leq \widehat{\ell_{\max}}$ since $\PAR_n
\subseteq \widehat{P}$, and hence the linear function $\widehat{\ell_{\max}} -
\ell(x)$ is nonnegative on $\EVEN_n$. Since $\widehat{P}$ has a
$G_{\parity}$-equivariant psd lift of size $d$, and since $\EVEN_n
\subseteq \widehat{P}$ one can show (using a simple generalization of Theorem \ref{thm:factorization-equivlift}) that we have an
equivariant certificate of nonnegativity of $\widehat{\ell_{\max}} - \ell$ on
$\EVEN_n$ of the form:
\[ \widehat{\ell_{\max}} - \ell(x) = \langle A(x), B \rangle \quad \forall x \in \EVEN_n \]
where $A$ satisfies the equivariance relation $A(g\cdot x) = \rho(g) A(x)
\rho(g)^T$ for all $x \in \EVEN_n$, $g \in G_{\parity}$.  Thus by Remark
\ref{rem:approx}, we know that $\widehat{\ell_{\max}} - \ell(x)$ is a
sum-of-squares of functions in a $G_{\parity}$-invariant subspace $V$ of
dimension $\leq d$. Thus using Lemma \ref{lem:parity-invariantsubspaces} below
it holds that $\widehat{\ell_{\max}} - \ell(x)$ is a sum-of-squares of
functions of degree $\leq k-1$ on $\EVEN_n$.

This is true for any facet-defining linear form $\ell$ of $\widehat{P}$ thus,
by the definition of the theta-body relaxation (cf. Equation \eqref{eq:SOSLV} with $V = \Pol_{\leq k-1}(\EVEN_n)$) we
have $\TH_{k-1}(\EVEN_n) \subseteq \widehat{P}$.
\if\siam1 \qquad \fi
\end{proof}

\paragraph{Exponential lower bounds} Before stating Lemma
\ref{lem:parity-invariantsubspaces} on invariant subspaces of $\cF(\EVEN_n,\RR)$
which is crucial for the proofs of the two theorems above, we show how by combining Theorem \ref{thm:lb-sym-parity} and  Proposition \ref{prop:parity-thetarank-lb} we get an exponential lower bound for equivariant psd lifts of the parity polytope. This is Theorem \ref{thm:main-parity-lb} from the introduction which we restate for convenience.

\begin{reptheorem}{thm:main-parity-lb}
Any $G_{\parity}$-equivariant psd lift of $\PAR_n$ for $n\geq 8$ must have
size $\geq \binom{n}{\lceil n/4 \rceil}$.
\end{reptheorem}
\begin{proof}
We apply Theorem \ref{thm:lb-sym-parity} above with $k=\lceil n/4 \rceil$. By Proposition \ref{prop:parity-thetarank-lb}, we
know that the theta body relaxation of order $\lceil n/4 \rceil-1$ is
\emph{not} exact. Thus this means that any $G_{\parity}$-equivariant psd
lift of $\PAR_n$ must have size $d \geq D_{n,\lceil n/4 \rceil}$. One can then
easily check from the definition of $D_{n,k}$ that when $n\geq 8$ we have
$D_{n,\lceil n/4 \rceil} \geq \binom{n}{\lceil n/4 \rceil}$.
\if\siam1 \qquad \fi
\end{proof}

\paragraph{Invariant subspaces of functions on $\EVEN_n$} 

To complete the proof of Theorems \ref{thm:lb-sym-parity} and
\ref{thm:approx-parity-sos}, it remains to show that low-dimensional invariant
subspaces of $\cF(\EVEN_n,\RR)$ are necessarily composed of low-degree polynomials.


We are interested in subspaces of $\cF(\EVEN_n,\RR)$ that are
$G_{\parity}$-invariant. Recall that $G_{\parity}$ is the group of
evenly signed permutations. Thus a subspace $V$ of $\cF(\EVEN_n,\RR)$ is
$G_{\parity}$-invariant if for any $f \in V$, and any $\epsilon \in
\{-1,+1\}^n$ such that $\prod_{i=1}^n \epsilon_i = 1$, and any $\sigma \in
\fS_n$ the function:
\[ x\mapsto f(\epsilon_1 x_{\sigma(1)},\dots,\epsilon_n x_{\sigma(n)}) \]
is also in $V$. Recall that an invariant subspace $V$ is called
\emph{irreducible} if it does not contain any nontrivial invariant subspace,
i.e., if $W$ is an invariant subspace of $V$, then $W = \{0\}$ or $W = V$.

It is clear that the subspaces $\Pol_k(\EVEN_n)$ are $G_{\parity}$-invariant.
The next result shows that these subspaces are actually irreducible.
Recall, for the statement to follow, that that the notation $D_{n,k}$ is defined in~\eqref{eq:Dnk}.

\begin{lemma}
\label{lem:parity-invariantsubspaces}
    Under the action of $G_{\parity}$, $\cF(\EVEN_n,\RR)$ decomposes into irreducible invariant subspaces as
    \[ \cF(\EVEN_n,\RR) = \Pol_0(\EVEN_n) \oplus \dots \oplus \Pol_{\lfloor n/2\rfloor}(\EVEN_n). \]
    Hence if $V$ is a $G_{\parity}$-invariant subspace of $\cF(\EVEN_n,\RR)$
    with $\dim(V) < D_{n,k}$ 
then  $V\subseteq \Pol_{\leq k-1}(\EVEN_n)$.
\end{lemma}
\begin{proof}
    It is easy to see that $\Pol_k(\EVEN_n)$ is invariant for each $k=0,1,\ldots,\lfloor n/2\rfloor$. It remains to show that 
    each of these is irreducible.
For any $k\neq \ell \in [n]$, let $\epsilon_{k,\ell} \in G_{\parity}$ be defined by $\epsilon_{k,\ell} = \diag(1,\dots,-1,\dots,-1,\dots,1)$ where all the entries are equal to $1$ except the entries in position $k$ and $\ell$ which are equal to $-1$. Given an element $p \in \cF(\EVEN_n,\RR)$ we denote by $(\id + \epsilon_{k\ell}) \cdot p$ the polynomial $p + \epsilon_{k\ell} \cdot p$. Observe that whenever $I\subset [n]$, then
    \[ (\id+\epsilon_{k\ell})\cdot x^I = x^I + \epsilon_{k\ell} \cdot x^I = \begin{cases} 2 x^I & \textup{if ($k\in I$ and $\ell\in I$) or ($k\notin I$ and 
                $\ell\notin I$)}\\
0 & \textup{if either exactly one of $k\in I$ and $\ell\in I$ occur}.\end{cases}\]
    
                Now fix some arbitrary $k < n/2$ (we deal with the case $n = 2k$ separately)
                and let $V$ be a (non-zero) invariant subspace of $\Pol_k(\EVEN_n)$. Let $p$ be a non-zero element of $V$.
                By the invariance of $V$ under the permutation action, we can assume that the coefficient of the monomial $x_1x_2\cdots x_k$ in $p$ is non-zero, and so $p$ is of the form: $p(x) = c x_1\cdots x_k + \sum_{|I|=k, I\neq\{1,2,\ldots,k\}} c_I x^I$ where $c\neq 0$.
                We will show that necessarily $V=\Pol_k(\EVEN_n)$. We first show that 
            \begin{equation}
                \label{eq:dualid}
        \left[\prod_{i=k+2}^n(\id+\epsilon_{k+1,i})\right]\cdot p(x)  
         = 2^{n-k-1}x_1x_2\cdots x_k.
     \end{equation}
     Once this is established we will know, by the $\fS_n$-invariance of $V$, that $V$ is equal to $\Pol_k(\EVEN_n)$.

         To establish~\eqref{eq:dualid}, first note that if $i\in \{k+2,k+3,\ldots,n\}$ then 
\[ (\id+\epsilon_{k+1,i})(x_1x_2\cdots x_k) = 2x_1x_2\cdots x_k \]
because
         neither of $x_i$ and $x_{k+1}$ appear in $x_1x_2\cdots x_k$. It remains to check that every other monomial of degree $k$
         is in the kernel of $\left[\prod_{i=k+2}^n(\id+\epsilon_{k+1,i})\right]$. Consider any other monomial $x^{I}$, i.e.~$I\subset \{1,2,\ldots,n\}$
         with $|I|=k$ and for which there is some $\ell\in I$ with $\ell \geq k+1$. Consider two cases, first the case where $k+1\notin I$. Then there is 
         some $\ell \geq k+2$ such that $\ell \in I$. But then $(\id+\epsilon_{k+1,\ell})\cdot x^I = 0$ and so since the $\epsilon_{i,j}$ commute, 
         $\left[\prod_{i=k+2}^{n}(\id+\epsilon_{k+1,i})\right]\cdot x^I = 0$. Now suppose $k+1\in I$. Then there is some $\ell \geq k+2$
         such that $\ell \notin I$. This is because if there were no such $\ell$ then we must have $I\supseteq\{k+1,k+2,\ldots,n\}$ which cannot have cardinality 
         $k$ since we assumed $k<n/2$. It then follows that $(\id+\epsilon_{k+1,\ell})\cdot x^I = 0$ and so that $\left[\prod_{i=k+2}^{n}(\id+\epsilon_{k+1,i})\right]\cdot x^I = 0$. 

         Finally consider the case when $n=2k$. In this situation since $V$ is invariant under the permutation action we can assume 
         \[ p(x) = c(x_1\cdots x_{n/2} + x_{n/2+1}\cdots x_n) + \sum_{
\substack{|I|=n/2, I\neq\{1,2,\ldots,n/2\}\\ I\neq \{n/2,n/2+1,\dots,n\}}} c_I (x^I + x^{I^c}).\]
         Applying the same argument as above, we see that 
         \[ \left[\prod_{i=k+2}^{n}(\id+\epsilon_{k+1,i})\right]\cdot p(x) = 2^{n-n/2-1}\left(x_1x_2\cdots x_{n/2} + x_{n/2+1}x_{n/2+2}\cdots x_{n}\right).\]
         Since the action of $\fS_n$ on 
\[ x_1x_2\cdots x_{n/2}+x_{n/2+1}x_{n/2+2}\cdots x_n \]
 generates a basis for $\Pol_{n/2}(\EVEN_n)$ we can 
         conclude that $V = \Pol_{n/2}(\EVEN_n)$.

The second part of the theorem is a direct consequence of Proposition
\ref{prop:lowdiminvsubspace} on low-dimensional invariant subspaces. When $n$
is odd note that $\dim \Pol_0(\EVEN_n) < \dim \Pol_1(\EVEN_n) < \dots < \dim
\Pol_{\lfloor n/2 \rfloor} (\EVEN_n)$ with $\dim \Pol_k(\EVEN_n) =
\binom{n}{k}$. Thus any invariant subspace $V$ of $\cF(\EVEN_n,\RR)$ with 
$\dim V < \dim \Pol_k(\EVEN_n) = \binom{n}{k} = D_{n,k}$ must be contained in
$V_{k-1}=\Pol_0(\EVEN_n) \oplus \dots \oplus \Pol_{k-1}(\EVEN_n)$ and thus
consists of polynomials of degree at most $k-1$.
In the case where $n$ is even we have $\dim \Pol_{n/2}(\EVEN_n) = \frac{1}{2} \binom{n}{n/2}$. 
Thus any invariant subspace $V$ with $\dim V <
\min\left(\binom{n}{k},\frac{1}{2} \binom{n}{n/2}\right) = D_{n,k}$ must be
contained in $\Pol_{\leq k-1}(\EVEN_n) = \Pol_0(\EVEN_n) \oplus \dots \oplus \Pol_{k-1}(\EVEN_n)$.
\if\siam1 \qquad \fi
\end{proof}


\section{The cut polytope}
\label{sec:cut}

\subsection{Definitions}

The maximum cut problem on a graph $G=(V,E)$ with $V=\{1,\dots,n\}$ and weights $w_{ij}$ for $ij \in E$ is the problem of labeling each vertex $i \in V$ with a label $x_i = +1$ or $x_i = -1$ in such a way that the total weight of edges connecting two vertices with a different label is maximized. This problem can be written as follows:
\begin{equation}
\label{eq:maxcut-quadform}
\begin{array}{ll}
\text{maximize} & \sum_{ij \in E} w_{ij} (1-x_i x_j)/2\\
\text{subject to} & x \in \{-1,1\}^n.
\end{array}
\end{equation}
Note that for a given labeling $x_i = \pm 1$ of vertices, the quantity $(1-x_i x_j)/2$ is equal to 1 if $i$ and $j$ have different labels, and 0 otherwise. The formulation \eqref{eq:maxcut-quadform} shows that the maximum cut problem is the problem of maximizing a \emph{quadratic} form on the hypercube $\{-1,1\}^n$. Using standard techniques, one can convert this problem into a \emph{linear} program, by working in a lifted space.
Indeed it is not hard to see that the problem \eqref{eq:maxcut-quadform} is equivalent to the problem below:
\begin{equation}
\label{eq:maxcut-lp}
\begin{array}{ll}
\text{maximize} & \sum_{ij \in E} w_{ij} (1-X_{ij})/2\\
\text{subject to} & X = xx^T \quad \text{for some } x \in \{-1,1\}^n.
\end{array}
\end{equation}
Note that the objective is now linear in the variable $X$. Define the cut polytope $\CUT_n$ as the convex hull of all outer products $xx^T$ for $x \in \{-1,1\}^n$:
\begin{equation}
 \label{eq:cutpolytope}
 \CUT_n = \conv \left\{ xx^T \; : \; x \in \{-1,1\}^n \right\}.
\end{equation}
The formulation \eqref{eq:maxcut-lp} shows that the maximum cut problem is a linear program over the cut polytope $\CUT_n$. Note that the cut polytope is a $n(n-1)/2$-dimensional polytope in the space $\S^n$ of $n\times n$ symmetric matrices.

\subsection{Sum-of-squares relaxations}
\label{sec:cutsos}

In this section we review the sum-of-squares relaxations of the cut polytope as described for example in \cite{laurent2003lower}. The construction we describe here actually applies to general polytopes $P$ of the form:
\begin{equation}
\label{eq:PconvxxT}
 P = \conv\left\{xx^T \; : \; x \in X\right\} \subset \S^n
\end{equation}
where $X$ is a finite set in $\RR^n$. The construction of the relaxation is very similar to the one described in the introduction, Section \ref{sec:intro-sos}, except that we work with quadratic forms instead of linear forms.
 Note that a polytope $P$ of the form \eqref{eq:PconvxxT} can be seen as the set of second-order moments of probability distributions on $X$:
\begin{equation}
\label{eq:X2nd}
\begin{aligned}
P = \Biggl\{ (E(e_{ij}))_{ij} \; : \; &  E \in \cF(X,\RR)^* \text{ where } \\
& \quad E(f) = \int_{X} f(x) d\mu(x) \text{ for some prob. measure $\mu$ on $X$} \Biggr\}
\end{aligned}
\end{equation}
where $e_{ij}$ is the element of $\cF(X,\RR)$ defined by $e_{ij}(x) = x_i x_j$. Given a subspace $V$ of $\cF(X,\RR)$ we can thus define the following relaxation of $P$ in the same way we did in Section \ref{sec:intro-sos}:
\begin{equation}
\label{eq:TH2VX}
\TH^{(2)}_V(X) := \Biggl\{ (E(e_{ij}))_{ij} \; : \; E \in \cF(X,\RR)^* \text{ where } E(1) = 1, \; E(f^2) \geq 0 \; \forall f \in V \Bigr\}.
\end{equation}
By comparing \eqref{eq:X2nd} and \eqref{eq:TH2VX} it is clear that $P\subseteq \TH_V^{(2)}(X)$.
One can show, like in Theorem \ref{thm:intro-soslifts} that the relaxation is tight $\TH^{(2)}_V(X) = P$ if for any quadratic form $q$ on $\RR^n$, $q_{\max}- q$ is $V$-sos, where $q_{\max} = \max_{x \in X} q(x)$.

When $X = \{-1,1\}^n$, the convex body $\TH^{(2)}_V(\{-1,1\}^n)$ is a relaxation of the cut polytope. When $V = \Pol_{\leq k}(\{-1,1\}^n)$ is the space of polynomials of degree at most $k$, we will denote the relaxation by $Q_k(\CUT_n)$:
\begin{equation}
\label{eq:defQkCUTTH2V}
Q_k(\CUT_n) = \TH^{(2)}_{V}(\{-1,1\}^n) \quad \text{ where } \quad V = \Pol_{\leq k}(\{-1,1\}^n).
\end{equation}
The relaxation $Q_k(\CUT_n)$ is known as the ``node-based'' relaxation and is usually described in the literature in terms of explicit moment matrices, see e.g., \cite{laurent2004semidefinite}. In fact if we use the basis of $\Pol_{\leq k}(\{-1,1\}^n)$ formed by square-free monomials of degree up to $k$, we get that $Q_k(\CUT_n)$ can be written as:
\begin{equation}
\label{eq:cutsosmom}
 Q_k(\CUT_n) = \left\{ z \in \RR^{\binom{n}{2}} \; : \; \exists (y_I)_{|I| \leq 2k} \text{ such that } 
\begin{array}{l}
y_{\emptyset} = 1\\
y_{ij} = z_{ij}, \;\; \forall i < j\\
\cM_k(y) \succeq 0
\end{array} \right\}
\end{equation}
where $(y_I)_{|I|\leq 2k}$ is a vector indexed by subsets $I\subseteq [n]$ of cardinality $\leq 2k$ and where $\cM_k(y)$ is the moment matrix associated to $y$ defined by:
\[ \cM_k(y)_{I,J} = y_{I \triangle J} \quad \forall I, J \subseteq [n], \; |I|, |J| \leq k \]
where $I\triangle J$ denotes symmetric difference. Laurent showed in \cite{laurent2003lower} that when $k \leq \lfloor n/2 \rfloor$ we have $Q_k(\CUT_n) \neq \CUT_n$.
\begin{theorem}[Laurent, \cite{laurent2003lower}]
\label{thm:lb-sos-cut}
For $k \leq \lfloor n/2 \rfloor$, the inclusion $\CUT_n \subset Q_k(\CUT_n)$ is strict.
\end{theorem}
Laurent conjectured in \cite{laurent2003lower} that the relaxation is actually tight for $k = \lceil n/2 \rceil$. This conjecture was proved recently in \cite{fawzi2015sparse}.

\subsection{Equivariant psd lifts of the cut polytope}
\label{sec:equivariantcut}

\paragraph{Symmetries of the hypercube and the cut polytope} 
Let
\[ C_n = \{-1,1\}^n \]
be the vertices of the hypercube in $\RR^n$. In Example \ref{ex:hypercube-regular-orbitope} we saw that $\conv(C_n)$ is the orbitope $\conv(N_{\cube} \cdot x_0)$ where $x_0 = \mathbf{1} \in \RR^n$ and $N_{\cube} := \{ \diag(\epsilon) \; : \; \epsilon \in \{-1,+1\}^n \}$. The full symmetry group of the hypercube is:
\[ G_{\cube} = \{ \epsilon h \; : \; \epsilon \in N_{\cube}, h \in \fS_n \} = N_{\cube} \rtimes \fS_n \]
where $\fS_n$ is the group of $n\times n$ permutation matrices. The group $G_{\cube}$ is known as the \emph{hyperoctahedral group} and consists of signed permutation matrices, i.e., permutation matrices where each nonzero entry is either $+1$ or $-1$.
Note that $G_{\cube}$ has the product structure described in Section \ref{sec:regular-orbitopes}, since $G_{\cube} = N_{\cube} \rtimes \fS_n$ and $\fS_n$ is the stabilizer of $x_0 = \mathbf{1} \in \RR^n$.


The group $G_{\cube}$ acts on the space of $n\times n$ symmetric matrices by congruence transformations as follows:
\begin{equation}
 \label{eq:actionGammacube-cut}
 g \cdot X := g X g^T \quad \forall g \in G_{\cube}, \; \forall X \in \S^n.
\end{equation}
It is easy to verify that $\CUT_n$ is invariant under this action of $G_{\cube}$.
Furthermore the sum-of-squares relaxations $Q_k(\CUT_n)$ are also $G_{\cube}$-invariant and using the same ideas as in Appendix \ref{sec:equivariancesoslift}, one can show that the psd lifts \eqref{eq:defQkCUTTH2V}-\eqref{eq:TH2VX} of $Q_k(\CUT_n)$ are $G_{\cube}$-equivariant.

\paragraph{Equivariant psd lifts of $\CUT_n$} In this section we are interested in psd lifts of $\CUT_n$ that are $G_{\cube}$-equivariant. Our main result is to show that any such equivariant lift must have exponential size. To do so we first prove that any $G_{\cube}$-equivariant lift of $\CUT_n$ must essentially be a low-degree sum-of-squares lift like $Q_k(\CUT_n)$. Then using the lower bound of Laurent \cite{laurent2003lower}, this implies an exponential lower bound on any $G_{\cube}$-equivariant psd lift of $\CUT_n$.

\begin{theorem}
\label{thm:lb-sym-cut}
Assume $\CUT_n$ has a $G_{\cube}$-equivariant $\S^d_+$-lift where $d < \binom{n}{k}$ for some $k\leq n/2$. Then the $(k-1)$'st sum-of-squares relaxation of $\CUT_{\lfloor n/2 \rfloor}$ is exact, i.e., $Q_{k-1}(\CUT_{\lfloor n/2 \rfloor}) = \CUT_{\lfloor n/2 \rfloor}$.
\end{theorem}
\begin{proof}
Assume we have a $G_{\cube}$-equivariant psd lift of $\CUT_n$ of size $d$. 
Using arguments very similar to Theorem \ref{thm:main-structure-product} (where linear forms are replaced by quadratic forms) we can show that there exists a $G_{\cube}$-invariant subspace $V$ of $\cF(C_n,\RR)$ with $\dim V \leq d$ such that for any quadratic form $q$ on $n$ variables with $q_{\max} := \max_{x \in C_n} q(x)$ we have:
\begin{equation}
 \label{eq:cube-sos-certificate}
 q_{\max}-q(x) = \sum_{i} f_i(x)^2 \quad \forall x \in C_n
\end{equation}
where each $f_i \in V$. In Lemma \ref{lem:cubegroupdecomp} (cf. below) we show that such an invariant subspace of dimension $d < \binom{n}{k}$, is composed entirely of polynomials of the form $g(x) + e_n(x) h(x)$ where $g$ and $h$ are polynomials of degree at most $k-1$ and $e_n(x) = x_1\cdots x_n$ is the $n$'th elementary symmetric polynomial.

We can use this to show that the $(k-1)$'st sos relaxation of $\CUT_{\lfloor n/2 \rfloor}$ is exact. Assume for simplicity that $n$ is even, $n=2m$ (the argument for $n$ odd is very similar). Let $q$ be an arbitrary quadratic form on $m$ variables. We will show that $q_{\max} - q(x)$ is a sum-of-squares of polynomials of degree at most $k$ on $C_m$ (i.e., it is $\Pol_{\leq k}(C_m)$-sos). Define $\hat{q}$ a quadratic form in $n=2m$ variables by:
\[ \hat{q}(x_1,\dots,x_n) = q(x_1,\dots,x_{m}). \]
Note that the polynomial $\hat{q}$ does not depend on $x_{m+1},\dots,x_n$ and note also that $\max_{x \in C_n} \hat{q}(x) = \max_{x \in C_m} q(x)$, i.e., $\hat{q}_{\max} = q_{\max}$.
From Equation \eqref{eq:cube-sos-certificate} we know that $\hat{q}_{\max}-\hat{q}(x)$ admits a sum-of-squares decomposition 
where each sos term lives in the subspace $V$, i.e., we have:
\begin{equation}
\label{eq:qhatmaxminusqhatx}
 \hat{q}_{\max} - \hat{q}(x) = \sum_{j} (\hat{g}_{j}(x) + e_n(x) \hat{h}_j(x))^2 \quad \forall x \in C_n
\end{equation}
where $\hat{g}_{j} \in \Pol_{\leq k-1}(C_n)$ and $\hat{h}_j \in \Pol_{\leq k-1}(C_n)$ and $e_n(x)$ is the $n$'th elementary symmetric polynomial $e_n(x) = x_1\cdots x_n$.
If we plug $x_{m+1} = x_1, x_{m+2}=x_2, \dots, x_{2m} = x_m$ in Equation \eqref{eq:qhatmaxminusqhatx} we get:
\begin{equation}
\label{eq:qmaxminusqx}
 q_{\max} - q(x) = \sum_{j} (g_j(x) + h_j(x))^2 \quad \forall x \in C_m
\end{equation}
where we used the fact that $e_n(x_1,\dots,x_m,x_1,\dots,x_m) = 1$ for all $x \in C_m$ and where we let 
\[ 
\begin{aligned}
g_j(x_1,\dots,x_m) &= \hat{g}_j(x_1,\dots,x_m,x_1,\dots,x_m)\\
h_j(x_1,\dots,x_m) &= \hat{h}_j(x_1,\dots,x_m,x_1,\dots,x_m)
\end{aligned}
\quad
\forall (x_1,\dots,x_m) \in C_m.
\]
Since $g_j, h_j \in \Pol_{\leq k-1}(C_n)$ it is easy to see that $\hat{g}_j, \hat{h}_j \in \Pol_{\leq k-1}(C_m)$. Equation \eqref{eq:qmaxminusqx} thus shows that $q_{\max}-q$ is $\Pol_{\leq k-1}(C_m)$-sos on $C_{m}$. Since $q$ was an arbitrary quadratic form on $m=n/2$ variables, this shows that the $(k-1)$'st level of the SOS hierarchy of $\CUT_{n/2}$ is exact.
\if\siam1 \qquad \fi
\end{proof}

Like for the parity polytope one can also state a result relating approximate equivariant lifts of the cut polytope and the sum-of-squares hierarchy. To state the result it is convenient to introduce the notion of a $(c,s)$-\emph{approximation} from the paper \cite{chan2013approximate}. Given two real numbers $c\leq s$ we say that an outer-approximation $\widehat{P}$ of $\CUT_n$ achieves a $(c,s)$-approximation of $\CUT_n$ if for any linear form $L$ on $\S^n$ such that $L_{\max} \leq c$ it holds that $\widehat{L_{\max}} \leq s$, where $L_{\max}$ and $\widehat{L_{\max}}$ are respectively the maximum of $L$ on $\CUT_n$ and $\widehat{P}$.
 We can now state the following theorem (we omit the proof since it is very similar to the arguments from the previous proofs):
\begin{theorem}
Assume $\widehat{P}$ is an outer-approximation of $\CUT_n$ which achieves a  $(c,s)$-approximation and which admits a $G_{\cube}$-equivariant psd lift of size $d$.
 If $d < \binom{n}{k}$ for some $k\leq n/2$ then the $(k-1)$'st sum-of-squares relaxation of $\CUT_{\lfloor n/2 \rfloor}$ is a valid $(c,s)$-approximation of $\CUT_{\lfloor n/2 \rfloor}$.
\end{theorem}
\paragraph{Exponential lower bound} If we combine Theorem \ref{thm:lb-sym-cut} with Laurent's $\lfloor n/2 \rfloor$ lower bound on the sum-of-squares hierarchy for the cut polytope we obtain the following exponential lower bound for $G_{\cube}$-equivariant psd lifts of $\CUT_n$. This is Theorem \ref{thm:main-cut-lb} from the introduction which we restate here for convenience.
\begin{reptheorem}{thm:main-cut-lb}
Any $G_{\cube}$-equivariant psd lift of $\CUT_{n}$ must have size $\geq \binom{n}{\lceil n/4 \rceil}$.
\end{reptheorem}
\begin{proof}
Let $m=\lfloor n/2 \rfloor$. We apply Theorem \ref{thm:lb-sym-cut} with $k=\lfloor m/2 \rfloor+1$. Laurent \cite{laurent2003lower} proved that $Q_{k-1}(\CUT_m) \neq \CUT_m$ and thus this means that any $G_{\cube}$-equivariant psd of lift of $\CUT_{n}$ must have size greater than or equal $\binom{n}{k} \geq \binom{n}{\lceil n/4 \rceil}$.
\end{proof}

\paragraph{A lemma on invariant subspaces of functions on the hypercube}

To complete the proof of Theorem \ref{thm:lb-sym-cut}, it remains to study low-dimensional invariant subspaces of $\cF(C_n,\RR)$. It is known that any function $f \in \cF(C_n,\RR)$ can be seen as a square-free polynomial of degree at most $n$. Let $\Pol_k(C_n)$ be the space of homogeneous square-free polynomials of degree $k$. Note that $\dim \Pol_k(C_n) = \binom{n}{k}$ and that:
\[ \cF(C_n,\RR) = \Pol_0(C_n) \oplus \dots \oplus \Pol_n(C_n). \]
We are interested in subspaces of $\cF(C_n,\RR)$ that are $G_{\cube}$-invariant. Recall that $G_{\cube}$ is the group of signed permutation matrices. Thus a subspace $V$ of $\cF(C_n,\RR)$ is $G_{\cube}$-invariant if for any $f \in V, \epsilon \in \{-1,+1\}^n, \sigma \in \fS_n$ the function:
\[ x\mapsto f(\epsilon_1 x_{\sigma(1)},\dots,\epsilon_n x_{\sigma(n)}) \]
is also in $V$. 

It is clear that the subspaces $\Pol_k(C_n)$ are $G_{\cube}$-invariant. The next result shows that these subspaces are actually irreducible under the action of $G_{\cube}$.

\begin{lemma}
    \label{lem:cubegroupdecomp}
    Under the action of $G_{\cube}$, $\cF(C_n,\RR)$ decomposes into irreducible invariant subspaces as
    \[ \cF(C_n,\RR) = \Pol_0(C_n) \oplus \dots \oplus \Pol_n(C_n). \]
    Furthermore, suppose $k < n/2$. Then $\Pol_{n-k}(C_n) \isom e_n(x)\Pol_{k}(C_n)$ where $e_n(x)=x_1\cdots x_n$ is the $n$'th elementary symmetric polynomial. Hence if $V$ is a $G_{\cube}$-invariant subspace 
    with $\dim(V) < \binom{n}{k}$ then every $f\in V$ has the form
    \[ f(x) = g(x) + e_n(x)h(x)\]
    where $g(x)$ and $h(x)$ have degree $\leq k-1$.
\end{lemma}
\begin{proof}
    It is easy to see that $\Pol_k(C_n)$ is $G_{\cube}$-invariant for each $k$. It remains to show that 
    each of these is irreducible. 
For any $k \in [n]$, let $\epsilon_k \in G_{\cube}$ be defined by $\epsilon_k = \diag(1,\dots,-1,\dots,1)$ where the $-1$ is in the $k$'th position. If $I \subseteq [n]$ we denote by $x^I$ the monomial $\prod_{i \in I} x_i$. Observe that for a given $k \in [n]$ and $I \subseteq [n]$ we have:
    \[ (\id+\epsilon_k)\cdot x^I = \begin{cases} 2 x^I & \textup{if $k\notin I$}\\
0 & \textup{otherwise}.\end{cases}\]
    In other words the action of $(\id+\epsilon_k)$ on $\cF(C_n,\RR)$ annihilates all monomials involving $x_k$. Similarly the action of
    $\prod_{k\in K}(\id+\epsilon_k)$ on $\cF(C_n,\RR)$ annihilates all monomials involving any $x_k$,  $k\in K$ since the $(\id+\epsilon_k)$ commute.
    Now fix some arbitrary $k$ and let $V$ be a (non-zero)
    invariant subspace of $\Pol_k(C_n)$. We will show that necessarily $V=\Pol_k(C_n)$. Since $V \neq \{0\}$, $V$ contains a nonzero square-free polynomial $p(x)$. By the invariance of $V$ under the permutation action, we can assume that the coefficient of the monomial $x_1x_2\cdots x_k$ in $p$ is non-zero, and so $p$ is of the form: $p(x) = c x_1\cdots x_k + \sum_{|I|=k, I\neq\{1,2,\ldots,k\}} c_I x^I$ with $c\neq 0$.
                Now note that by the previous observation we have:
    \[
        \left[\prod_{i=k+1}^n(\id+\epsilon_{i})\right]\cdot p(x)
         = 2^{n-k}x_1x_2\cdots x_k.\]
         Hence $V$ is a subspace containing $x_1x_2\cdots x_k$ and hence, since $V$ is invariant under the permutation action, it contains every square-free monomial of degree $k$. It follows that 
         $V = \Pol_k(C_n)$ and so $\Pol_k(C_n)$ is irreducible. 

To show the second part of the theorem, we use Proposition \ref{prop:lowdiminvsubspace} from the introduction. Indeed note that $\dim \Pol_k(C_n) = \binom{n}{k}$. Thus Proposition \ref{prop:lowdiminvsubspace} says that if $V$ is an invariant subspace and $\dim V < \binom{n}{k}$ with $k \leq n/2$ then necessarily $V$ is contained in the direct sum
\[ \bigoplus_{i=0}^{k-1} (\Pol_i(C_n) \oplus \Pol_{n-i}(C_n)). \]
Thus this means that any $f \in V$ can be decomposed as:
\[ f = \sum_{i=0}^{k-1} g_i + g_{n-i} \]
where $g_i \in \Pol_i(C_n)$ and $g_{n-i} \in \Pol_{n-i}(C_n)$. Now note that for $i<n/2$, $\Pol_{n-i}(C_n)\isom e_n(x)\Pol_{i}(C_n)$ because multiplication by $e_n(x)$ is an involution of $\cF(C_n,\RR)$ that sends square-free polynomials of degree $i$ to square-free polynomials of degree $n-i$. Thus we can write $g_{n-i}(x) = e_n(x) h_i(x)$ for some $h_i \in \Pol_i(C_n)$. Thus we get that
\[ f(x) = \sum_{i=0}^{k-1} g_i(x) + e_n(x) h_i(x) = g(x) + e_n(x) h(x) \]
where $\deg g \leq k-1$ and $\deg h \leq k-1$.
\if\siam1 \qquad \fi
\end{proof}


\section{Conclusion}

In this paper we presented a general framework to study equivariant psd lifts of \emph{orbitopes} using tools from representation theory. We studied two particular examples in detail, namely the parity polytope and the cut polytope and we derived strong lower bounds for sizes of equivariant psd lifts.

\paragraph{Other examples of orbitopes}
There are clearly other examples of orbitopes that one could study using the framework developed in this paper. In a follow-up paper \cite{polygonsequivariant} (see also  \cite{fawzi2015sparse}) we looked at the case of regular polygons in the plane. We showed that regular $N$-gons in the plane admit equivariant psd lifts of size $O(\log N)$ and we established a matching lower bound. Our construction of the equivariant psd lift relies on finding a \emph{sparse} sum-of-squares certificate for the facet inequalities of the regular $N$-gon.
In the more recent paper \cite{fawzi2015sparse} we also give efficient equivariant psd lifts of certain trigonometric cyclic polytopes. Other examples of orbitopes that are interesting to study are for example the permutahedron and the perfect matching polytope.

\paragraph{Non-polyhedral orbitopes} Also, even though our paper was mainly concerned with orbitopes constructed from finite groups $G$, the structure theorems of Section \ref{sec:orbitopes} also apply when $G$ is a compact subgroup of $GL(\RR^n)$ and $P$ is not necessarily polyhedral. For example, an interesting orbitope initially studied in \cite{sanyal2011orbitopes} is the convex hull of $SO(n)$, where $SO(n)$ is the set of $n\times n$ special orthogonal matrices. In \cite{jamesSOn} it is shown that this convex body is a spectrahedron and an explicit semidefinite description is given. 
The structure theorem of this paper can also be applied to this example and shows that one can study equivariant psd lifts of $\conv(SO(n))$ by studying invariant subspaces of $\cF(SO(n),\RR)$ (the space of real-valued functions on $SO(n)$).

It is interesting that, for the purpose of obtaining lower bounds, our study of the parity polytope already implies an exponential lower bound on the size of equivariant lifts of $\conv(SO(n))$: indeed it is known, by a celebrated result of Horn \cite{horn1954doubly} that $SO(n)$ projects onto $\PAR_n$ via the diagonal map, i.e., $\diag(SO(n)) = \PAR_n$. This can be used to show that any psd lift of $\conv(SO(n))$ equivariant
with respect to an appropriate symmetry group can be turned into a $G_{\parity}$-equivariant psd lift of $\PAR_n$ of the same size. Hence our exponential lower bound on the size of $G_\parity$-equivariant psd lifts of $\PAR_n$ automatically gives 
an exponential lower bound on the size of certain equivariant psd lifts of $\conv(SO(n))$. The details of this argument are given in \cite{jamesSOn}.

\clearpage
\newpage

\appendix


\section{Proof of Theorem \ref{thm:intro-equivariantsoslifts}: equivariance of sum-of-squares lifts when $V$ is $G$-invariant.}
\label{sec:equivariancesoslift}

Let $X$ be a finite set in $\RR^n$ and assume $X$ is invariant under the action of a group $G\subseteq GL(\RR^n)$. We also assume that $\conv(X)$ is full-dimensional. In this appendix, we show that the psd lift \eqref{eq:sosliftV}-\eqref{eq:SOSLV} is $G$-equivariant when the subspace $V$ is chosen to be $G$-invariant.

Assume $V$ is a subspace of $\cF(X,\RR)$ such that any valid linear inequality on $\conv(X)$ has a sum-of-squares certificate from $V$.
Then we saw in Section \ref{sec:intro-sos} that we have the following description of $\conv(X)$:
\begin{equation}
\label{eq:appendix-sosliftV}
\begin{aligned}
\conv(X) = \Biggl\{ (E(e_1),\dots,E(e_n)) \; : \; & E \in \cF(X,\RR)^* \text{ where }\\
& E(1) = 1, \; E(f^2) \geq 0 \; \forall f \in V \Bigr\}.
\end{aligned}
\end{equation}
where for each $i$, $e_i \in \cF(X,\RR)$ is the function defined by $e_i(x) = x_i$. Equation \eqref{eq:appendix-sosliftV} expresses $\conv(X)$ as a psd lift of size $d=\dim V$. Indeed the last constraint on $E$ in \eqref{eq:appendix-sosliftV} is equivalent to saying that the bilinear form $H_{E}$ on $V$ defined by 
\[
H_E:V\times V \rightarrow \RR, \quad H_E(f_1,f_2) = E(f_1 f_2)
\]
 is positive semidefinite. Thus if we identify $\S^{d}$ with bilinear forms on $V$ (by fixing a basis) then Equation \eqref{eq:appendix-sosliftV} can be rewritten as:
\begin{equation}
\label{eq:appendix-sosliftV-2}
\conv(X) = \pi(\S^{d}_+ \cap L)
\end{equation}
where 
\begin{itemize}
\item $L \subset \S^{d}$ is the affine subspace
\[ L := \{ H_{E} \; : \; E \in \cF(X,\RR)^*, E(1)=1 \}, \]
i.e., $L$ the image of the affine space $\{E \in \cF(X,\RR)^*, E(1)=1\}$ under the linear map $E \mapsto H_{E}$;
\item and $\pi$ is the linear map, which given a bilinear form of the form $H_{E}$ for some $E \in \cF(X,\RR)^*$, returns the $n$-tuple $(E(e_1),\dots,E(e_n)) \in \RR^n$ (this linear map $\pi$ is well-defined when $P$ is full-dimensional, since one can show in this case that the functions $e_i$ are all in $\linspan\{f^2:f \in V\}$).
\end{itemize}

We now proceed to show that the psd lift \eqref{eq:appendix-sosliftV-2} satisfies the definition of $G$-equivariance, where $G$ is the automorphism group of $X$.

Since $G$ acts on $\cF(X,\RR)$, it also acts on the dual space $\cF(X,\RR)^*$ as follows: If $E \in \cF(X,\RR)^*$ then we let
\[ (g\cdot E)(f) := E(g^{-1}\cdot f) \quad \forall f \in \cF(X,\RR). \]
Equivariance of the lift \eqref{eq:appendix-sosliftV-2} now follows from the following main lemma:
\begin{lemma}
Given $E \in \cF(X,\RR)^*$ we have for any $g \in G$ and any $f,h \in V$:
\begin{equation}
 \label{eq:MM}
 H_{g\cdot E}(f_1,f_2) = H_{E}(g^{-1}\cdot f_1,g^{-1}\cdot f_2).
\end{equation}
\end{lemma}
\begin{remark}
One can interpret the identity \eqref{eq:MM} in matrix terms as follows: Given $g\in G$, let $\theta(g)$ be the $d\times d$ matrix which corresponds to the linear map $f\in V \mapsto g\cdot f$. Define $\rho(g) = \theta(g^{-1})^T$. Then identity \eqref{eq:MM} is the same as:
\begin{equation}
\label{eq:MMM}
 H_{g\cdot E} = \rho(g) H_{E} \rho(g)^T
\end{equation}
where $H_{g\cdot E}$ and $H_{E}$ are interpreted as symmetric matrices of size $d$. $\lozenge$
\end{remark}
\begin{proof}
The following sequence of equalities proves the claim:
\[\begin{aligned}
 H_{g\cdot E}(f_1,f_2) = (g\cdot E)(f_1 f_2) 
&= E(g^{-1}\cdot (f_1 f_2))\\
&\overset{(*)}{=} E((g^{-1} \cdot f_1)(g^{-1} \cdot f_2))
 = H_{E}(g^{-1}\cdot f_1, g^{-1} \cdot f_2). \end{aligned} \]
In Equality (*) we used the fact that the action of $G$ on $\cF(X,\RR)$ satisfies:
\[ g\cdot (f_1 f_2) = (g\cdot f_1) (g\cdot f_2) \]
which can be easily seen since for any $x \in X$ we have:
\[ g\cdot (f_1 f_2)(x) = (f_1 f_2)(g^{-1} \cdot x) = f_1(g^{-1} \cdot x)  f_2(g^{-1} \cdot x) = (g\cdot f_1)(x) (g\cdot f_2)(x). \]
\if\siam1 \qquad \fi 
\end{proof}

Using this lemma it is easy to check that the lift \eqref{eq:appendix-sosliftV-2} satisfies the definition of $G$-equivariance.


\section{Proof of Theorem \ref{thm:factorization-equivlift}: factorization theorem for equivariant psd lifts}
\label{sec:proof_factorization_theorem_equivariant_lifts}

In this appendix we give the proof of the factorization theorem for equivariant psd lifts, Theorem \ref{thm:factorization-equivlift}. Before doing so, we first prove the following factorization theorem for general positive semidefinite lifts from  \cite{gouveia2011lifts}:

\begin{theorem}
\label{thm:generalpsdliftfactorization}
Let $P$ be a polytope in $\RR^n$ and let $X$ be any subset of $P$. Assume $P$ has a psd lift of size $d$, i.e., $P = \pi(\S^d_+ \cap L)$ where $L$ is an affine subspace of $\S^d$. Let $A:X\rightarrow \S^d_+ \cap L$ be any map such that $\pi(A(x)) = x$ for all $x \in X$. Then for any linear function $\ell \in (\RR^n)^*$, there exists $B(\ell) \in \S^d_+$ such that:
\[ \ell_{\max} - \ell(x) = \langle A(x), B(\ell) \rangle \quad \forall x \in X \]
where $\ell_{\max} := \max_{x \in P} \ell(x)$.
\end{theorem}
\begin{proof}
The proof is mainly an application of SDP duality. Let $\ell \in (\RR^n)^*$. Let $Y_0 \in L$ and let $L_0 = L - Y_0$ be the linear subspace parallel to $L$. The following two SDPs are dual to each other:
\begin{equation}
\label{eq:primaldualellmax}
\begin{array}{ll}
\text{max} & (\ell \circ \pi)(Y)\\
\text{s.t.} & Y \in \S^d_+\\
            & Y - Y_0 \in L_0
\end{array}
\qquad
\qquad
\begin{array}{ll}
\text{min} & -\langle H, Y_0 \rangle\\
\text{s.t.} & -\ell \circ \pi = B + H\\
            & B \in \S^d_+, \; H \in L_0^{\perp}
\end{array}
\end{equation}
We assume for simplicity that the intersection of $L$ with the interior of $\S^d_+$ is nonempty, so that strong duality holds (the general case can also be handled). In this the case the value of the two SDPs above is equal to $\ell_{\max}$. 
Let $B,H$ be the optimal points of the dual (minimization) problem in \eqref{eq:primaldualellmax}. From dual feasibility we have $-\ell \circ \pi = B + H$ and so since $\ell_{\max} = -\langle H, Y_0 \rangle$ we get that:
\[ \ell_{\max} - \ell \circ \pi = B + H - \langle H, Y_0 \rangle. \]
Evaluating this equality at $A(x)$, for any $x \in X$ we get:
\[ \ell_{\max} - \langle \ell, x \rangle = \langle B, A(x) \rangle \]
where we used the fact that $\langle \ell\circ \pi, A(x) \rangle = \langle \ell, x \rangle$ (since $\pi(A(x)) = x$) and $\langle H, A(x) \rangle - \langle H, Y_0 \rangle = 0$.
\if\siam1 \qquad \fi
\end{proof}

We now prove the factorization theorem for equivariant psd lifts, Theorem \ref{thm:factorization-equivlift} from Section \ref{sec:orbitopes} which we restate below for convenience (the proof belows looks at the case where $X$ consists of a single orbit of $G$, i.e., $X = G\cdot x_0$ which are the main focus of this paper -- the proof however can be easily extended to the case where $X$ consists of more than a single orbit, see e.g., \cite[Theorem 2]{gouveia2011lifts}).
\begin{reptheorem}{thm:factorization-equivlift}
Let $G$ be a finite group acting on $\RR^n$ and let $X = G\cdot x_0$ where $x_0 \in \RR^n$.
 Assume $\conv(X) = \pi(\S^d_+ \cap L)$ is a $G$-equivariant psd lift of $\conv(X)$ of size $d$, i.e., there is a homomorphism $\rho:G\rightarrow GL(\RR^d)$ such that conditions (i) and (ii) of Definition \ref{def:equivariantpsdlift} hold. Then there exists a map $A:X \rightarrow \S^d_+$ with the following properties:
\begin{itemize}
\item[(i)] For any linear form $\ell$ on $\RR^n$ there exists $B(\ell) \in \S^d_+$ such that if we let $\ell_{\max} := \max_{x \in X} \ell(x)$ we have:
\[ \ell_{\max} - \ell(x) = \langle A(x), B(\ell) \rangle \quad \forall x \in X. \]
\item[(ii)] The map $A$ satisfies the following equivariance relation:
\[ A(g\cdot x) = \rho(g)A(x)\rho(g)^T \quad \forall x \in X, \; \forall g \in G. \]
In particular if $H$ denotes the stabilizer of $x_0$, then we have:
\begin{equation}
\label{eq:Ax0stabilizer2}
 A(x_0) = \rho(h) A(x_0) \rho(h)^T \quad \forall h \in H.
\end{equation}
\end{itemize}
Furthermore, the representation $\rho:G\rightarrow GL(\RR^d)$ can be taken to be orthogonal, i.e., $\rho(g) \in O(\RR^d)$ for all $g \in G$.
\end{reptheorem}
\begin{proof}
We first treat the case where the stabilizer of $x_0$ is $\{1_G\}$. In this case the proof is almost trivial: Let $A(x_0)$ be any point in $\S^d_+ \cap L$ such that $\pi(A(x_0)) = x_0$. Define, for any $g \in G$, $A(g\cdot x_0) := \rho(g) A(x_0) \rho(g)^T$. Note that $A(g\cdot x_0) \in \S^d_+ \cap L$ by the definition of an equivariant psd lift (Definition \ref{def:equivariantpsdlift}). Also note that 
\begin{equation}
\label{eq:appB-Agx0-proof}
 \pi(A(g\cdot x_0)) = \pi(\rho(g)A(x_0)\rho(g)^T) \overset{(a)}{=} g\cdot \pi(A(x_0)) \overset{(b)}{=} g\cdot x_0
\end{equation}
where in $(a)$ we used the definition of equivariant psd lift, and in $(b)$ we used the fact that $\pi(A(x_0)) = x_0$. Since $X = G\cdot x_0$, Equation \eqref{eq:appB-Agx0-proof} shows that $\pi(A(x)) = x$ for all $x \in X$. Thus we can use Theorem \ref{thm:generalpsdliftfactorization} with our choice of map $A$ to show that Property (i) of the statement holds. Note that Property (ii) holds by construction of the map $A$.

We now treat the general case. Let $H$ be the stabilizer of $x_0$. Let $A_0$ be any point in $\S^d_+ \cap L$ such that $\pi(A_0) = x_0$ and define:
\[ A(x_0) = \frac{1}{|H|} \sum_{h \in H} \rho(h) A_0 \rho(h)^T. \]
Now we extend the map $A$ to the whole $X$ by letting $A(x) = \rho(g) A(x_0) \rho(g)^T$ where $g$ is any element of $G$ such that $g\cdot x_0$. Note that this is well-defined since if $g\cdot x_0 = g'\cdot x_0$, we have, with $h = g^{-1} g' \in H$:
\[ \begin{aligned}
\rho(g') A(x_0) \rho(g')^T &= \rho(gh) A(x_0) \rho(gh)^T \\
&= \rho(g) \rho(h) A(x_0) \rho(h)^T \rho(g)^T = \rho(g) A(x_0) \rho(g)^T.
\end{aligned}
\]
It is easy to verify, like in the previous case, that the map we just defined satisfies $\pi(A(x)) = x$ for all $x \in X$. Thus by applying Theorem \ref{thm:generalpsdliftfactorization} with our choice of map $A$ we get that Property (i) of the statement holds. Also Property (ii) holds by construction of the map $A$.

Note that we can assume $\rho(g)$ to be orthogonal for any $g \in G$, by a simple change of basis: By Proposition \ref{prop:changebasis-orthrep} let $Q$ be an invertible matrix such that $Q \rho(g) Q^{-1}$ is orthogonal. By letting $\hat{\rho}(g) = Q \rho(g) Q^{-1}$, $\hat{A}(x) = Q A(x) Q^T$, $\hat{B} = Q^{-T} B Q^{-1}$ we have $\langle \hat{A}(x), \hat{B} \rangle = \langle A(x), B \rangle$, $\hat{A}(g\cdot x) = \hat{\rho}(g) \hat{A}(x) \hat{\rho}(g)^T$ and $\hat{\rho}(g) \hat{\rho}(g)^T = I_d$ which is what we need.
\if\siam1 \qquad \fi
\end{proof}

\bibliography{../../bib/psd_lifts}

\if\siam1
\bibliographystyle{siam}
\else
\bibliographystyle{alpha}
\fi

\end{document}